\documentclass[a4j,12pt]{amsart}

\usepackage{amsmath}
\usepackage{amsbsy}
\usepackage{amssymb}
\usepackage{amscd}
\usepackage{euscript}

\makeatletter
\renewcommand{\c@equation}{\c@thm}
\makeatother

\usepackage[dvips]{graphicx, color}
\usepackage{slashbox}
\newcommand{\kankaku}{\setlength{\itemsep}{3mm}}

\newcommand{\bbc}{{\mathbb C}}

\newcommand{\bbq}{{\mathbb Q}}
\newcommand{\bbr}{{\mathbb R}}

\newcommand{\bbz}{{\mathbb Z}}

\newcommand{\be}{{\beta}}

\newcommand{\gam}{{\gamma}}

\newcommand{\del}{{\delta}}

\newcommand{\im}{{\operatorname {Im}}}
\newcommand{\re}{{\operatorname {Re}}}

\newcommand{\gl}{{\operatorname{GL}}}

\newcommand{\Z}{\bbz}
\newcommand{\Q}{\bbq}
\newcommand{\R}{\bbr}
\newcommand{\C}{\bbc}

\newcommand{\beq}{\begin{equation}\begin{aligned}}
\newcommand{\eeq}{\end{aligned}\end{equation}}
\newcommand{\beginp}{\begin{proof}[証明]}

\newcommand{\beqn}{\begin{eqnarray}}
\newcommand{\eeqn}{\end{eqnarray}}

\newcommand{\op}{\operatorname}

\newcommand{\suchh}{\ \Bigg|\ }

\newcommand{\beqnn}{\begin{eqnarray*}}
\newcommand{\eeqnn}{\end{eqnarray*}}
\newcommand{\beqnna}{\begin{eqnarray}\begin{array}}
\newcommand{\eeqnna}{\end{array}\end{eqnarray}}
\newcommand{\bsub}{\begin{subarray}}
\newcommand{\esub}{\end{subarray}}

\newtheorem{thm}{Theorem}[section]
\newtheorem{lem}[thm]{Lemma}

\newtheorem{prop}[thm]{Proposition}

\newtheorem{rem}[thm]{Remark}        

\title{On the Mordell--Weil group of the elliptic curve $\boldsymbol{y^2=x^3+n}$}
\author{Yasutsugu Fujita}
\author{Tadahisa Nara}

\date{}

\address[Y. Fujita]
{Department of Mathematics, 
College of Industrial Technology, Nihon University, 2-11-1 Shin-ei, 
Narashino, Chiba 275--8576, Japan}

\address[T. Nara]
{Mathematical Institute, Tohoku University, Sendai 980-8578, Japan}

\begin{document}

\begin{abstract}
We study an infinite family of Mordell curves 
(i.e. the elliptic curves in the form  
$y^2=x^3+n$, $n\in\Z$) 
over $\Q$ 
with three explicit integral points. 
We show that the points are independent 
in certain cases. 
We describe how to compute bounds of 
the canonical heights of the points. 
Using the result we show that 
any pair in the three points  
can always be a part of a basis 
of the free part of the Mordell--Weil group. 
\end{abstract}

\maketitle

\section{Introduction}

Let $E$ be an elliptic curve over a number field $K$. It is known that 
the set of 
rational points $E(K)$ is a finitely generated abelian group by 
the Mordell--Weil theorem. 
If the absolute value of the discriminant of $E$ is not large, 
we can practically use Cremona's program `mwrank'. 
However there is no known algorithm 
which determines the structure of $E(K)$ even if $K=\Q$.
The difficulties 
come from the free part of the group. 
We are interested in 
the families of elliptic curves of which we can at least 
partially determine the structure 
of the Mordell--Weil group, 
that is, the families which have explicit points which 
can be in a system of generators of the Mordell--Weil group. 
In the paper \cite{duquesne1} Duquesne considered 
an infinite family of elliptic curves in the form $y^2=x^3-nx$. 
He showed that the curves in the family 
have two explicit integral points which can 
always be in a system of generators. 
Recently, the first author and Terai (\cite{FT}) 
generalized Duquesne's theorem on generators 
and showed that the same is true for infinitely many
binary forms $n=n(k,l)$ in $\Z[k,l]$.
In this paper we consider an infinite family 
of elliptic curves in the form of $y^2=x^3+n$ with 
three explicit integral points. 

Let 
$a,b$ be integers and 
\begin{equation}
\label{eq:Eab-defn}
E_{a,b}:y^2=x^3+a^6+16b^6 
\end{equation}
the elliptic curve over $\Q$. 
We put 
\begin{equation}
\label{eq:3points}
P_1=(-a^2, 4b^3),\ 
P_2=(2ab,a^3+4b^3),\ P_3=(-2ab,a^3-4b^3). 
\end{equation}
Then it is easy to see that they are in $E_{a,b}(\Q)$. 
In this paper we prove the following theorem. 

\begin{thm}
\label{main}
Assume that $a,b$ are relatively prime integers with $a,b\geq3$ 
such that 
$a^6+16b^6$ is square-free, $ab$ is odd and 
$b$ is divisible by $3$ but not by $9$. 
Then the rank of the Mordell--Weil group $E_{a,b}(\Q)$ is at least $3$ 
and 
any pair of two points $\{P_i,P_j\}$ $(i=1,2,3,\ i\neq j)$ 
can always be in a system of generators of 
$E_{a,b}(\Q)$. 
\end{thm}
\begin{rem}
\upshape
If $n$  
is square-free and not equal to $1$,
the elliptic curve $y^2=x^3+n$ 
has no rational torsion points by 
\cite[Theorem 5.3]{knappB1}. 
Therefore $P_1$, $P_2$, $P_3$ are non-torsion 
in the situation of Theorem \ref{main}. 
\end{rem}

We prove Theorem \ref{main}
along similar lines to 
Duquesne's (\cite{duquesne1}). 
First 
we compute bounds of the canonical heights of 
$P_i$'s $(i=1,2,3)$. 
For that purpose we use the decomposition of 
the canonical height 
into the sum of local heights. 
The non-archimedean part of 
local heights is 
computed by Silverman's algorithm. 
To compute the archimedean part of 
local heights, 
we use Tate's series. Then the archimedean part 
is in the form 
$
\log|x(P_i)|+
\frac{1}{4}\sum_{n=0}^{\infty} 4^{-n}\log|z(2^nP_i)|
$, 
where $z(P)$ is a certain function. 
But to do this, it is necessary 
that the $x$-coordinates of the points on the curve 
are away from zero. 
To deal with the difficulty, we shift the elliptic curve 
in the direction of $x$-axis. 
We set the shifting width by functions of $a,b$ 
such as $2a^2+4b^2, 3a^2+4b^2$. 
Then the computation of the bounds comes down 
to a maximum-minimum problem of elementary functions. 
Further we find 
in our case $z(P)$ above 
is bounded independently of $a,b$ and $P$. 
Thanks to this, we obtain an upper bound and a lower bound 
whose difference is a constant.

On the other hand we find a uniform lower bound 
of the canonical height. It is a bound independent 
of the points on the elliptic curve. 
This is computed by 
Cohen's algorithm. 

Next using those bounds 
we estimate the lattice indices of 
$\{P_i,P_j\}$ $(i,j=1,2,3, \ i\neq j)$ 
(for the definition of the lattice index 
see Section \ref{sec:lattice index}). 
The key theorem is Siksek's theorem 
which comes from the theory of quadratic forms. 

The goal of the proof of Theorem \ref{main} is 
to show that the lattice indices equal $1$. 
By the estimete using Siksek's theorem 
we can show that the lattice indices are less than $5$. 
On the other hand we prove that the lattice indices 
are not divisible by 2, 3 by an argument of the descent. 
This completes the proof.

The organization of this paper is as follows. 
In Section \ref{sec:preliminaries} 
we review basic notations of elliptic curves. 
We also review the canonical height and the local height function. 
In Section \ref{sec:comp-canon-height} we compute bounds of 
the canonical heights of $P_1, P_2, P_3$. 
In Section \ref{sec:unif lower} 
we compute an uniform lower bound of the canonical height. 
In Section \ref{sec:lattice index} we estimate the lattice indices 
by applying Siksek's theorem to the results of 
Sections \ref{sec:comp-canon-height} and \ref{sec:unif lower}. 
In Section \ref{sec:indep} 
we prove that the lattice indices are not divisible by 2, 3 
by an argument of the descent. 
Then we complete the proof of Theorem \ref{main}. 
Further we prove that the family of the elliptic curves 
satisfying the condition of Theorem \ref{main} 
is an infinite family. 
Finally in Section \ref{sec:uni z'} we compute 
the bounds of $z(P)$, which are used in Section \ref{sec:comp-canon-height}.

\section{Preliminaries}
\label{sec:preliminaries}

The standard symbols $\Q$, $\R$, $\C$ and $\Z$ 
will denote respectively
the set of rational, real and complex numbers 
and the rational integers.
We denote the discrete valuation on 
$\Z$ at the prime $p$ by $v_p(\cdot)$.  
We denote the set of all places of a 
number field $K$ by $M_K$.

Throughout this paper, we assume that 
$a,b\in \Z$, $a, b \geq3$, $\gcd(a,b)=1$ and 
$m=a^6+16b^6$.

As usual we write the Weierstrass equation 
for elliptic curves $E$ over a number field $K$ as 
\begin{equation}
\label{eq:elliptic-defn}
E : y^2+a_1xy+a_3y=x^3+a_2x^2+a_4x+a_6 \;
(a_1,a_2,a_3,a_4,a_6\in K).
\end{equation}
Since the characteristic of $K$ is not equal to $2$, 
by completing the square of the left-hand side we have 
\begin{equation}
\label{eq:2y=}
(2y+a_1x+a_3)^2=4x^3+b_2x^2+2b_4x+b_6, 
\end{equation}
where
\begin{equation}
\label{eq:b_i}
\begin{aligned}
& b_2=a_1^2+4a_2,\ \ \ b_4=2a_4+a_1a_3, \ \ \ b_6=a_3^2+4a_6, \\
& b_8=a_1^2a_6+4a_2a_6-a_1a_3a_4+a_2a_3^2-a_4^2. 
\end{aligned} 
\end{equation}
Further, by putting 
\begin{equation*}
c_4=b_2^2-24b_4,\ \ \ c_6=-b_2^2b_8+36b_2b_4-216b_6,  
\end{equation*}
we have 
\begin{equation*}
\{108(2y+a_1x+a_3)\}^2=(36x+3b_2)^3-27c_4(36x+3b_2)-54c_6. 
\end{equation*}
We also define the discriminant of $E$ as 
\begin{equation}
\label{eq:disc}
\Delta = -b_2^2b_8-8b_4^3-27b_6^2+9b_2b_4b_6. 
\end{equation}
Using the form (\ref{eq:b_i}), we can write 
\begin{equation}
\label{eq:map2}
x(2P)=\frac{x^4-b_4x^2-2b_6x-b_8}{4x^3+b_2x^2+2b_4x+b_6}
\end{equation}
for $P=(x,y)\in E$.

Next we define the canonical height, 
which is a powerful tool to consider the arithmetic of elliptic curves. 
Let $E$ be an elliptic curve over $\Q$ and $P=(x,y)\in E(\Q)$. 
If $x=n/d$ and $\gcd(n,d)=1$, we define the na\"{i}ve height of 
$P$ by $h(P)=\max \{\log|n|, \log|d|\}$ (\cite[p. 202]{aec})
and the canonical height of $P$ by 
\begin{equation*}
\hat{h}(P)=\lim_{n \rightarrow \infty} \frac{h(2^nP)}{4^n}
\end{equation*}
(\cite[p. 248]{aec}).  

\begin{rem}
\upshape
In our definition the value of $\hat{h}$
is twice of those 
in \cite{aec}, \cite{cohen0} and \cite{sil1}.
\end{rem}

The canonical height has the following properties. 
\begin{itemize}
\item
$\hat{h}(P)=0$  if and only if $P$ is a torsion point.

\item
$\hat{h}(kP)=k^2\hat{h}(P)$ for all $P \in E(\Q)$ and all $k \in \Z$.

\item
$\hat{h}$ is a quadratic form on $E$.

\end{itemize}
For details see also \cite[Chapter VIII Section 9]{aec}.

Our computations of the canonical height is done by using 
the local height. We recall the existence of the local height function 
as follows.

\begin{thm}
(N$\acute{e}$ron, Tate, \cite[p. 341]{sil1})
\label{thm:Neron}
Let $K$ be a number field, $v$ a place and 
$K_v$ its completion respect to 
an absolute value $|\cdot|_v$.
Let $E$ be the elliptic curve over $K$ 
given by $(\ref{eq:elliptic-defn})$. 
Then there exists a unique function 
$\hat{\lambda}_v : E(K_v)\setminus{O} \rightarrow \R$
which has the following three properties. 
\begin{itemize}
\item[(1)]
For all P $\in E(K_v)$ with $2P \neq O$, 
\begin{equation*} 
\hat{\lambda}_{v}(2P)=4\hat{\lambda}_{v}(P)-
2\log |2y(P)+a_1x(P)+a_3|_v. 
\end{equation*}
\item[(2)]
The limit 
\begin{math} 
\lim_{\substack{P\rightarrow O \\ \text{$v$-adic}}}
(\hat{\lambda}_v(P)-\log |x(P)|_v) 
\end{math}
exists.  
\item[(3)]
$\hat{\lambda}_v$ is 
bounded on any $v$-adic open subset of $E(K_v)$ disjoint from $O$.
\end{itemize}
\end{thm}

The function $\hat{\lambda}_v$ above is called the 
{\emph {local height function}}.  If we have to 
specify the elliptic curve, we may use the notation 
such as $\hat{\lambda}_{E,v}$.  
The canonical height can be decomposed as 
the sum of local heights. 
The sum of the local heights for all archimedean (resp. non-archimedean) 
places is 
called the archimedean (resp. non-archimedean) part of the canonical height 
and denoted by $\hat{h}_f(P)$ (resp. $\hat{h}_{\infty}(P)$).
We only consider
the case $K=\Q$ and in this situation, 
\begin{equation} 
\label{eq:decom-global-local}
\hat{h}(P)=
\hat{h}_f(P)+ \hat{h}_{\infty}(P)=
\sum_{p : {\rm prime}} \hat{\lambda}_p(P) + \hat{\lambda}_{\infty}(P). 
\end{equation}

Let $d \in K$ and 
\begin{align*}
E': (y')^2+{a_1}'x'y'+{a_3}'y' & = (x')^3+{a_2}'(x')^2+{a_4}'x'+{a_6}' 
\end{align*}
the elliptic curve obtained by making 
the substitution 
\begin{equation}
\label{eq:defn-x'}
x'=x+d,\ y'=y 
\end{equation}
in (\ref{eq:elliptic-defn}). Then 
\begin{equation}
\label{eq:a_i'}
\begin{aligned}
{a_1}' & =a_1,\ {a_2}'=a_2-3d,\ {a_3}'=a_3-da_1,\\
{a_4}' & =a_4-2da_2+3d^2,\ {a_6}'=a_6-da_4+d^2a_2-d^3.
\end{aligned}
\end{equation}

Now let $P\in E(K_v)$ and $P'=\left(x(P)+d, y(P)\right)\in E'(K_v)$.
It is clear that the map 
$E(K_v)\ni P\mapsto P'\in E'(K_v)$ is a group isomorphism. 

\begin{lem}
\label{lem:local inv}
In the situation above, 
we have $\hat{\lambda}_{E,v}(P)=\hat{\lambda}_{E',v}(P')$. 
\end{lem}
\begin{proof}
To see this, it is sufficient to show that the 
function $f : E'(K_v) \rightarrow \R$ 
defined by $f(P')=\hat{\lambda}_{E,v}(P)$ satisfies the three properties 
of $\hat{\lambda}_v$ in Theorem \ref{thm:Neron}.

The property $(1)$ follows from the equality 
\begin{equation*}
2y'+{a_1}'x'+{a_3}'
=2y+a_1(x+d)+{a_3}-da_1
=2y+a_1x+a_3. 
\end{equation*}
For the property $(2)$, we have 
\begin{equation*}
\begin{aligned}
\lim_{\substack{P'\rightarrow O' \\ \text{$v$-adic}}}
\{f(P')- \log |x'(P')|_v\} 
& = \lim_{\substack{P\rightarrow O \\ \text{$v$-adic}}}
\{\hat{\lambda}_{E,v}(P)-\log |x(P)+d|_v\} \\
& = \lim_{\substack{P\rightarrow O \\ \text{$v$-adic}}}
\left\{\hat{\lambda}_{E,v}(P)-\log |x(P)|_v - 
\log \left|\frac{x(P)+d}{x(P)}\right|_v\right\} \\
& = \lim_{\substack{P\rightarrow O \\ \text{$v$-adic}}}
\left\{\hat{\lambda}_{E,v}(P)-\log |x(P)|_v 
- \log \left|1+\frac{d}{x(P)}\right|_v\right\} \\ 
& = \lim_{\substack{P\rightarrow O \\ \text{$v$-adic}}}
\{\hat{\lambda}_{E,v}(P)-\log |x(P)|_v\}. 
\end{aligned}
\end{equation*}
The property $(3)$ is clearly satisfied. 
\end{proof}

\section{Computing the canonical height}
\label{sec:comp-canon-height}

Let $E_{a,b}$ be the elliptic curve (\ref{eq:Eab-defn})
and $P_1,P_2,P_3$ the rational points on $E_{a,b}$ 
defined in (\ref{eq:3points}). 

\begin{prop}
\label{prop1}
If $ab$ is odd, $v_3(b)=1$ and $m$ is square-free, 
then the canonical heights of the points $P_1$, $P_2$, $P_3$ 
have the following bounds
\begin{align*}
& \frac{1}{3} \log m -0.7441 < \hat{h}(P_1) 
< \frac{1}{3} \log m +0.5409, \\
& \frac{1}{3} \log m -0.7579 < \hat{h}(P_2) 
< \frac{1}{3} \log m +1.0515, \\
& \frac{1}{3} \log m -0.5113 
< \hat{h}(P_3) < \frac{1}{3} \log m +0.5665.
\end{align*}
\end{prop}
\begin{proof}
[Proof of Proposition \ref{prop1}]
We use the decomposition (\ref{eq:decom-global-local})
to estimate the canonical height. 
We first estimate the archimedean part 
$\hat{h}_{\infty}(P)$=$\hat{\lambda}_{\infty}(P)$ 
by using Tate's series with Silverman's shifting trick (\cite{sil1}).

Let $E$ be the elliptic curve defined by (\ref{eq:elliptic-defn}). 
For $P \in E(\R)$, 
we put 
\begin{equation}
\begin{aligned}
t & = t(P):=1/x(P), \\
z & = z(P):=1-b_4t^2-2b_6t^3-b_8t^4, \\
w & = w(P):=4t+b_2t^2+2b_4t^3+b_6t^4, 
\end{aligned}
\end{equation}
where $b_2$, $b_4$, $b_6$, $b_8$ are as in (\ref{eq:b_i}).
Note that we have 
\begin{math}
x(2P)={z(P)}/{w(P)}.
\end{math}
By the property of the local height (Theorem \ref{thm:Neron} (1)) we have 
\begin{equation*} 
\hat{\lambda}_{\infty}(2P)=4\hat{\lambda}_{\infty}(P)-
2\log |2y(P)+a_1x(P)+a_3|.
\end{equation*}
Then using (\ref{eq:2y=}), we have 
\begin{equation*}
\begin{aligned}
\hat{\lambda}_{\infty}(2P)-\log |x(2P)| 
& = 
4\hat{\lambda}_{\infty}(P)-
2\log |2y(P)+a_1x(P)+a_3|
-\log|x(2P)|\\
& = 
4\hat{\lambda}_{\infty}(P)-
\log |4x(P)^3+b_2x(P)^2+2b_4x(P)+b_6| \\
& \quad -\log|z(P)/w(P)|\\
& = 
4\{\hat{\lambda}_{\infty}(P)-\log |x(P)|\}-\log |z(P)| .
\end{aligned}
\end{equation*}

Putting $\mu(P):=\hat{\lambda}_{\infty}(P)-\log |x(P)|$, 
\begin{equation*} 
\mu(2P)=4\mu(P)-\log |z(P)|.
\end{equation*} 
So if we ignore the convergence, we have
\begin{equation*} 
\mu(P)=\frac{1}{4}\sum_{n=0}^{\infty} 4^{-n}\log|z(2^nP)|. 
\end{equation*}
In fact, by Tate's theorem (\cite[Theorem 1.2]{sil1}) 
if there is $\epsilon>0$ such that 
$|x(P)|>\epsilon$ for all $P\in E(\R)$, 
then for any $P\in E(\R)$, $\log |z(2^nP)|$ 
is bounded independently of $n$ and therefore 
\begin{equation*} 
\hat{\lambda}_{\infty}(P)=\log|x(P)|+
\frac{1}{4}\sum_{n=0}^{\infty} 4^{-n}\log|z(2^nP)|. 
\end{equation*}

For $d \in \Q$ and $P \in E(\R)$, 
the point $P'=(x(P)+d,y(P))$ is on the curve 
\begin{equation}
\label{eq:E'}
E': (y')^2+{a_1}'x'y'+{a_3}'y' = (x')^3+{a_2}'(x')^2+{a_4}'x'+{a_6}', 
\end{equation}
where 
\begin{equation*}
\begin{aligned}
{a_1}' & =a_1,\ {a_2}'=a_2-3d,\ {a_3}'=a_3-da_1,\\
{a_4}' & =a_4-2da_2+3d^2,\ {a_6}'=a_6-da_4+d^2a_2-d^3 
\end{aligned}
\end{equation*}
as we saw in (\ref{eq:a_i'}).
We similarly put 
\begin{equation}
\begin{aligned}
\label{eq:z'}
t' & = t'(P'):=1/x'(P'), \\
z' & = z'(P'):=1-b'_4(t')^2-2b'_6(t')^3-b'_8(t')^4, \\
w' & = w'(P'):=4t'+b'_2(t')^2+2b'_4(t')^3+b'_6(t')^4, 
\end{aligned}
\end{equation}
where $b'_2$, $b'_4$, $b'_6$, $b'_8$ are 
the values obtained 
by replacing $a_1,\ldots,a_6$ by  $a_1',\ldots,a_6'$ 
in (\ref{eq:b_i}).  

The reason why we make this substitution is that 
we obtain the Weierstrass model to which we can apply 
Tate's theorem above. We call this the 
\emph{shifting trick} following Silverman.   
\begin{figure}[htbp]
 \begin{minipage}{0.49\hsize}
  \begin{center}
   \includegraphics[width=76mm]{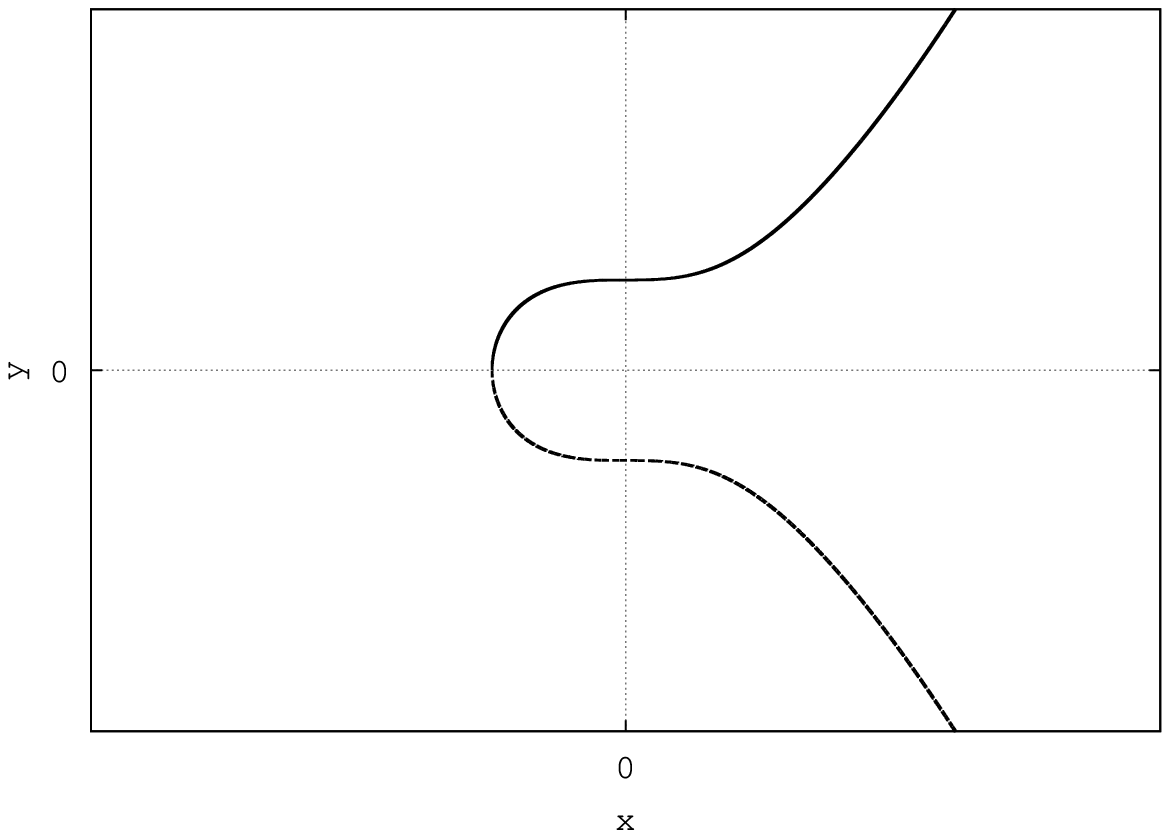}
  \end{center}
  \caption{\protect\\ $y^2+a_1xy+a_3y = \protect\\ 
x^3+a_2x^2+a_4x+a_6$}
  \label{fig:one}
 \end{minipage}
 \begin{minipage}{0.49\hsize}
  \begin{center}
   \includegraphics[width=76mm]{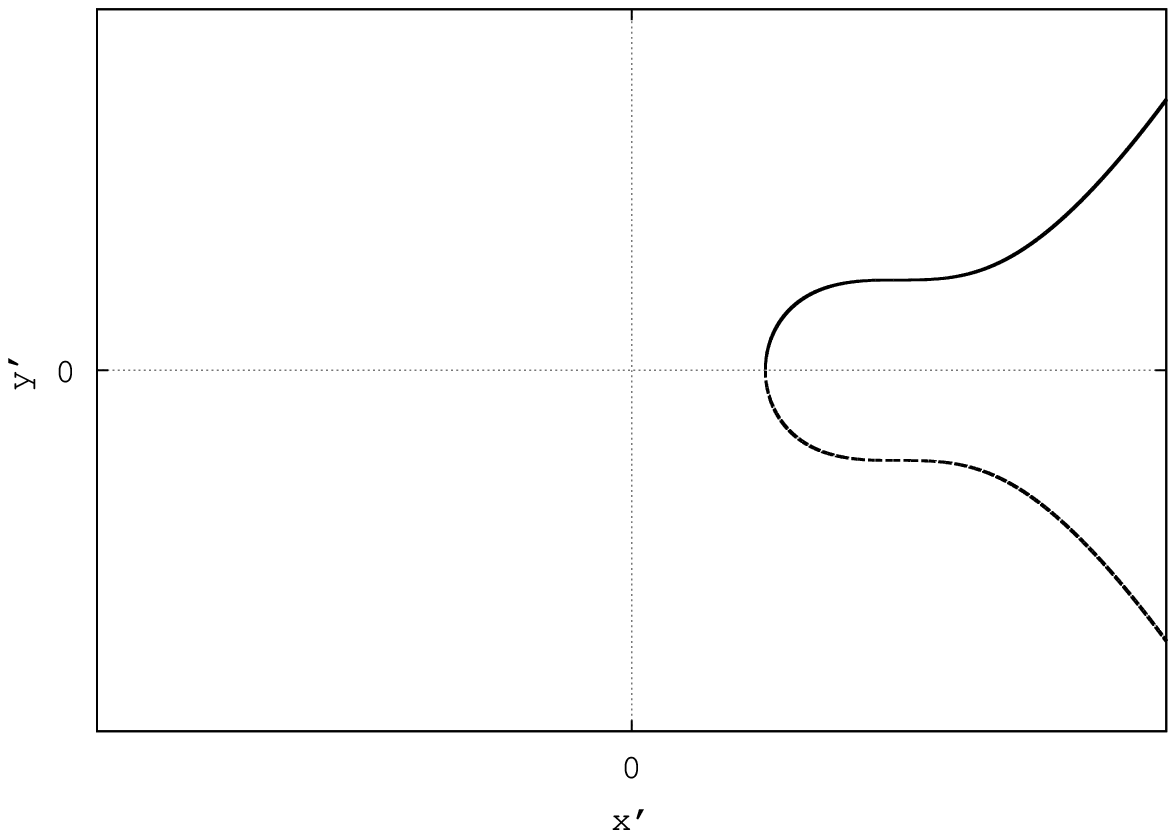}
  \end{center}
  \caption{\protect\\ $(y')^2+a'_1x'y'+a'_3y'= \protect\\ 
(x')^3+a'_2(x')^2+a'_4x'+a'_6$}
  \label{fig:two}
 \end{minipage}
\end{figure}

Now let $E=E_{a,b}$. If $P\in E(\R)$, 
then $x(P)\geq -m^{1/3}$. 
So if we take $d$ such that $d>m^{1/3}$, then 
$x'(P')=x(P)+d \geq -m^{1/3}+d >0$. Therefore 
the assumption of Tate's result 
is satisfied and we have
\begin{equation*} 
\hat{\lambda}_{E',\infty}(P')=\log|x'(P')|+
\frac{1}{4}\sum_{n=0}^{\infty} 4^{-n}\log|z'(2^nP')|. 
\end{equation*}
On the other hand, 
by Lemma \ref{lem:local inv} 
\begin{math} 
\hat{\lambda}_{E',\infty}(P') 
=\hat{\lambda}_{E,\infty}(P).
\end{math}
So we have 
\begin{equation*} 
\hat{\lambda}_{\infty}(P)
=\hat{\lambda}_{E,\infty}(P) 
=\log|x'(P')|+
\frac{1}{4}\sum_{n=0}^{\infty} 4^{-n}\log|z'(2^nP')|. 
\end{equation*}

Let us compute $\hat{\lambda}_{\infty}(P_2)$ by this formula, 
taking $d=2a^2+4b^2$. 
Then the condition $d>m^{1/3}$ is 
clearly satisfied. 
So we have 
\begin{equation}
\label{eq:series P_2}
\hat{\lambda}_{\infty}(P_2)=\log |x'(P_2')| +\frac{1}{4} \sum_{n=0}^{\infty}
4^{-n}\log |z'(2^{n}P_2')| .
\end{equation}
We put $X=a/b$ and compute 
$x'(P_2'),\ z'(P_2'),\ z'(2P_2'),\ z'(4P_2')$. 
Since $x(P_2)=2ab$, the following hold. 
\begin{itemize}
\kankaku
\item 
$x'(P_2')=2ab+2a^2+4b^2$ (see (\ref{eq:3points}) 
for the coordinate of $P_2$) 
\item
\begin{math}
z'(P_2')=
(X^8 - 2 X^7 + 2 X^6 + 8 X^5 + 2 X^4 + 16 X^3 + 16 X^2 - 32 X + 32)/(2 X^8 + 8 X^7 + 28 X^6 + 56 X^5 + 98 X^4 + 112 X^3 + 112 X^2 + 64 X + 32) 
\end{math}
\item
\begin{math}
z'(2P_2')=
(X^{32} + 4 X^{31} + 2 X^{30} - 32 X^{29} + 124 X^{28} - 64 X^{27} + 144 X^{26} + 864 X^{25} + 288 X^{24} + 1344 X^{23} + 9408 X^{22} + 2688 X^{21} + 8256 X^{20} + 34176 X^{19} + 16512 X^{18} + 59904 X^{17} + 237600 X^{16} + 119808 X^{15} + 132096 X^{14} + 546816 X^{13} + 264192 X^{12} + 344064 X^{11} + 2408448 X^{10} + 688128 X^{9} + 589824 X^{8} + 3538944 X^{7} + 1179648 X^{6} - 2097152 X^{5} + 8126464 X^{4} - 4194304 X^{3} + 1048576 X^{2} + 4194304 X + 2097152)\\/(2 X^{32} - 16 X^{31} + 64 X^{30} - 96 X^{29} - 96 X^{28} + 864 X^{27} - 1376 X^{26} + 256 X^{25} + 6800 X^{24} - 13120 X^{23} + 16640 X^{22} + 35200 X^{21} - 70400 X^{20} + 141184 X^{19} + 121472 X^{18} - 217600 X^{17} + 878624 X^{16} + 123904 X^{15} - 570368 X^{14} + 2603008 X^{13} - 1312768 X^{12} - 753664 X^{11} + 6422528 X^{10} - 6127616 X^{9} + 1884160 X^{8} + 8126464 X^{7} - 12845056 X^{6} + 8912896 X^{5} + 5767168 X^{4} - 16777216 X^{3} + 16777216 X^{2} - 8388608 X + 2097152)
\end{math}
\item
\begin{math}
z'(4P_2')=
(X^{128} - 8 X^{127} + 2 X^{126} + 384 X^{125} + \cdots 
+38685626227668133590597632)\\/
(2 X^{128} + 32 X^{127} + 208 X^{126} + 448 X^{125} + \cdots 
+38685626227668133590597632)
\end{math}
\end{itemize}
In the computation of $z'(P_2')$, $z'(2P_2')$, $z'(4P_2')$, 
we used the software PARI/GP (Version 2.3.4) \cite{pari}. 
The commands
\begin{verbatim}
allocatemem(64000000);
E=ellinit([0,0,0,0,m]);
Ed=ellchangecurve(E,[1,-d,0,0]);
zd=1-Ed.b4*(x+d)^-2-2*Ed.b6*(x+d)^-3-Ed.b8*(x+d)^-4;
zd_a=subst(zd,m,a^6+16*b^6);
zd_b=subst(zd_a,d,2*a^2+4*b^2);

zd_1P2p=subst(zd_b,x,2*a*b);
zd_1P2=subst(subst(zd_1P2p,a,X),b,1)
zd_2P2p=subst(zd_b,x,ellpow(E,[2*a*b,a^3+4*b^3],2)[1]);
zd_2P2=subst(subst(zd_2P2p,a,X),b,1)
zd_4P2p=subst(zd_b,x,ellpow(E,[2*a*b,a^3+4*b^3],4)[1]);
zd_4P2=subst(subst(zd_4P2p,a,X),b,1)
\end{verbatim}
compute them. 

Since 
${x'(P_2')^3}/{m}$, $z'(P_2')$, $z'(2P_2')$, $z'(4P_2')$ 
are functions of $X$, by elementary calculus 
we can compute their maximum and minimum. 
So we can find the following bounds.
\begin{equation}
\label{zbounds}
\begin{aligned}
 \frac{1}{3}\log (4m) < & \log x'(P_2') < \frac{1}{3}\log (57.2218701m), \\
 -0.6637015 < &4^{-1}\log z'(P_2') < 0, \\
 -0.0433217 < &4^{-2}\log z'(2P_2') < 0.1396289, \\
 -0.0363430 < &4^{-3}\log z'(4P_2') \leq 0 .
\end{aligned}
\end{equation}

For example, to have the bounds of $\log x'(P_2')$ as above 
it suffices to show 
\begin{equation*}
4 < \frac{x'(P_2')^3}{m}\ \left(= \frac{(2X^2+2X+4)^3}{X^6+16}\right) 
< 57.2218701. 
\end{equation*}
We type (the above codes are needed)
\begin{verbatim}
f2p=(2*a*b+2*a^2+4*b^2)^3/(a^6+16*b^6);
f2=subst(subst(f2p,a,X),b,1);
df2=deriv(f2,X);
df2n=numerator(df2);
fac2=factor(df2n*1.)
\end{verbatim}
to find the factorization of the numerator of 
$((2X^2+2X+4)^3/(X^6+16))'$.  We do not write the 
output here, but it turns out that the fourth 
root ($X=1.6484223 \cdots$) is the only positive root. 
Then $x'(P_2')^3/m = 57.22187008 \cdots$  
by the following command. 
\begin{verbatim}
subst(f2,X,-subst(fac2[4,1],X,0))
\end{verbatim}
Since $\lim_{X\rightarrow 0}(2X^2+2X+4)^3/(X^6+16) = 4$ and
$\lim_{X\rightarrow \infty}(2X^2+2X+4)^3/(X^6+16) = 8$, 
we have the bounds for $\log x'(P_2')$ as above.

To compute the bounds of 
$z'(P_2')$, $z'(2P_2')$, $z'(4P_2')$, 
we proceed similarly.  
Note that 
if only $a,b$ are real numbers, 
$d=2a^2+4b^2 > m^{1/3}$ is satisfied. 
Then $\log|z'(2^nP_2')|$ 
has a finite value by Tate's theorem. 
So the denominators of 
$z'(P_2')$, $z'(2P_2')$, $z'(4P_2')$ do not have real roots. 
For $z'(P_2')$, we type 
\begin{verbatim}
dzd_1P2=deriv(zd_1P2);
dzd_1P2n=numerator(dzd_1P2);
facdz_1P2=factor(dzd_1P2n*1.)
\end{verbatim}
to find the factorizations of the derivative of 
$z'(P_2')$.  Then
we substitute all the values of the positive roots 
to $4^{-1}\log z'(P_2')$ by the commands \\
\\
\verb|4^-1*log(subst(zd_1P2,X,-subst(facdz_1P2[|$\Box_1$\verb|,1],X,0)))|\\
\\
where $\Box_1=4$ because only the fourth root is positive. 
Since 
$\lim_{X\rightarrow 0}4^{-1}\log z'(P_2')=0$ and 
$\lim_{X\rightarrow \infty}4^{-1}\log z'(P_2')=4^{-1}\log (1/2)$
$=-0.173286795 \cdots$, 
we have the bounds for $4^{-1}\log z'(P_2')$. 

For $z'(2P_2')$ and $z'(4P_2')$ we type \\
\\
\verb|dzd_2P2=deriv(zd_2P2);|\\
\verb|dzd_2P2n=numerator(dzd_2P2);|\\
\verb|facdz_2P2=factor(dzd_2P2n*1.)|\\
\verb|4^-2*log(subst(zd_2P2,X,-subst(facdz_2P2[|$\Box_2$\verb|,1],X,0)))|\\
\\
\verb|dzd_4P2=deriv(zd_4P2);|\\
\verb|dzd_4P2n=numerator(dzd_4P2);|\\
\verb|facdz_4P2=factor(dzd_4P2n*1.)|\\
\verb|4^-3*log(subst(zd_4P2,X,-subst(facdz_4P2[|$\Box_3$\verb|,1],X,0)))|\\
\\
where $\Box_2=6,7,8$ and $\Box_3=8,9,12,13,16$ 
since they are all the indices which correspond to the positive roots 
in factorizations up to multiplicity. 
Since 
$
\lim_{X\rightarrow 0}\log z'(2P_2')
$
$
=\lim_{X\rightarrow 0}\log z'(4P_2')
=0$, 
$
\lim_{X\rightarrow \infty}4^{-2}\log z'(2P_2')=4^{-2}\log (1/2)$
=$-0.0433216 \cdots$ 
and 
$
\lim_{X\rightarrow \infty}4^{-3}\log z'(4P_2')=4^{-3}\log (1/2)$
=$-0.0108304 \cdots$, 
by comparing the values obtained 
we have the bounds for $4^{-2}\log z'(2P_2')$ and $4^{-3}\log z'(4P_2')$. 

For the estimate of the remaining terms $z'(2^nP_2')$ ($n \geq 3$) 
we use 
the following two lemmas, 
which we shall prove in Section \ref{sec:uni z'}. 

\begin{lem}
\label{remboundu}
Let $d=2a^2+4b^2$ or $d=3a^2+4b^2$. 
Then  
\begin{math}
 z'(P')< 120.531634 
\end{math}
for any $P \in E_{a,b}(\R)$. 
\end{lem}
\begin{lem}
\label{remboundl}
\begin{itemize}
\setlength{\itemsep}{1.5mm}
\item[{\upshape (1)}]
If $d=2a^2+4b^2$, then  
\begin{math}
0.062326 < z'(P') 
\end{math}
for any $P\in E_{a,b}(\R)$. 
\item[{\upshape (2)}]
If $d=3a^2+4b^2$, then 
\begin{math}
0.038068 < z'(P') 
\end{math}
for any $P\in E_{a,b}(\R)$. 
\end{itemize}
\end{lem}
\begin{rem}
In general there is 
Silverman's bound of $z'(P')$ $($\cite[Lemma 4.1]{sil1}$)$, 
which gives a bound dependent on $a,b$. 
In our case we find that there is 
a bound of $z'(P')$ independent of $a,b$. 
\end{rem}

We continue the proof of Proposition \ref{prop1}. 
Since $(1/4)\sum_{n=3}^\infty 4^{-n}=1/192$, 
we have 
\begin{equation}
\label{eq:n>=3}
\frac 1{192}\log(0.062326) 
< \frac{1}{4}\sum_{n=3}^{\infty}4^{-n}\log z'(2^nP_2') 
< 
\frac 1{192}
\log(120.531634).
\end{equation}
By (\ref{eq:series P_2}), 
(\ref{zbounds}) and (\ref{eq:n>=3}), 
we have 
\begin{equation*}
\frac{1}{3}\log m-0.295724 <\hat{\lambda}_{\infty}(P_2)
<\frac{1}{3}\log m +1.513566. 
\end{equation*}

To compute the non-archimedean part $\hat{h}_f(P_2)$, 
we use Lemma \ref{lem:finite part}, 
which is proved in the next subsection. 
Recall $P_2 = (2ab, a^3+4b^3)$. So $\alpha, \beta, \delta$ 
in Lemma \ref{lem:finite part} 
correspond to $2ab, a^3+4b^3, 1$ respectively. 
Therefore 
\begin{equation*} 
\hat{h}_f(P_2) = 
-\frac{2}{3}\log 2. 
\end{equation*}
Since 
$\hat{h}(P_2) =\hat{\lambda}_{\infty}(P_2)+\hat{h}_f(P_2)$,  
we have 
\begin{equation*}
\frac{1}{3} \log m -0.7579 
< \hat{h}(P_2) 
< \frac{1}{3} \log m +1.0515 .
\end{equation*}

We can estimate $\hat{h}(P_1)$, $\hat{h}(P_3)$ similarly 
by taking $d=3a^2+4b^2, 2a^2+4b^2$ respectively.
\end{proof}

\begin{rem}
The shifting width is not necessary to be $3a^2+4b^2, 2a^2+4b^2$. 
We choose the width which give good enough bounds. 
We do not have an idea to determine the width 
which give the best bound. 
\end{rem}
\subsection{Non-archimedean part}

In this subsection 
we compute the
non-archimedean part of the canonical height, which was required 
in the proof of Proposition \ref{prop1}. 
To do this, we use \cite[THEOREM 5.2]{sil1}. 
But the Weierstrass equation of the elliptic curve 
to which we apply this theorem 
needs to be minimal at $p$ to compute $\hat{\lambda}_p$. 
Let $n \in \Z$ be sixth power free 
and $E$ the elliptic curve $y^2=x^3+n$. 
Then the Weierstrass equation of $E$ 
is global minimal if and only if $n\not \equiv 16\pmod {64}$ 
(\cite[Corollary 5.6.4]{connel1}). 
Since in Theorem \ref{main} we assume that $a^6+16b^6$ is square-free, 
we may assume that $n\not \equiv 16\pmod {64}$, 
and that $y^2=x^3+n$ is global minimal. 
Further we may assume that $v_2(n)=0$, $v_3(n)=0$. 
The reason is explained by the next lemma. 
\begin{lem}
Let $a, b \in \Z$ with $\gcd(a,b)=1$. 

$(1)$ If $n=a^6+16b^6$, then 
\begin{equation*} 
v_3(n)=0. 
\end{equation*}

$(2)$ If $n=a^6+16b^6$ and $n\not \equiv 16\pmod {64}$, then  
\begin{equation*} 
v_2(n)=0.
\end{equation*}
\end{lem}
\begin{proof}
$(1)$ For any $k \in \Z$, we have 
$
(3k\pm1)^2 = 9k^2 \pm 6k +1
 \equiv 1\pmod 3. 
$
So since $\gcd(a,b)=1$, 
we have $(a^6,b^6) \equiv (1, 0)$ , $(1, 1)$ or $(0, 1) \pmod 3$. 
In all the cases $n \not\equiv 0 \pmod 3$.

$(2)$ We assume that $v_2(n)\neq 0$, 
and deduce a contradiction. 
Since $v_2(n)\neq 0$, $a \equiv 0 \pmod2$.
Since $\gcd(a,b)=1$, 
$b \not \equiv 0 \pmod2$. Since $b$ is odd, $b^3$ is odd 
and we can write $b^3=2k+1$. 
Then we have 
$n=a^6+16(2k+1)^2 \equiv 64k^2+64k+16 \equiv 16 \pmod{64}$, 
since $a^6 \equiv 0 \pmod{64}$.
This is a contradiction. 
\end{proof}

In the next two lemmas 
we compute $\hat{\lambda}_2(P)$, $\hat{\lambda}_3(P)$ assuming 
that $v_2(n)=0$, $v_3(n)=0$ respectively. 
\begin{lem}
\label{lem:nonarc2}
Let $n (\not\equiv 16\pmod{64})$ be sixth power free and 
$E$ the elliptic curve $y^2=x^3+n$ over $\Q$.  
Let $P=(\alpha/\delta^2, \beta/\delta^3)$ 
$(\alpha,\beta,\delta \in \Z,\ \delta>0,\ 
\gcd(\alpha,\delta)=\gcd(\beta,\delta)=1)$ be 
a rational point on $E$.
Assume $v_2(n)=0$. 
If $v_2(\alpha)=0$, then $\hat{\lambda}_2(P)=2v_2(\delta)\log2$.
If $v_2(\alpha)\neq 0$, then $\hat{\lambda}_2(P)=-\frac{2}{3}\log2$. 
\end{lem}
\begin{proof}
Since $n \not\equiv 16\pmod{64}$, $y^2=x^3+n$ is global minimal. 
So we compute $\hat{\lambda}_2(P)$ following 
the algorithm (\cite[p.354, SUBROUTINE in THEOREM 5.2]{sil1}).

For the general Weierstrass equation (\ref{eq:elliptic-defn}) 
and a point $P$ on it, we put 
\begin{math} 
x:=x(P),\ y:=y(P). 
\end{math}
Further 
we define $A$, $B$, $C$, $\Lambda$ for $P$ as follows. 
\begin{equation}
\label{eq:ABCL} 
\begin{aligned} 
& A:=v_p(3x^2+2a_2x+a_4-a_1y),\ B:=v_p(2y+a_1x+a_3), \\
& C:=v_p(3x^4+b_2x^3+3b_4x^2+3b_6x+b_8), \\
& \Lambda:=\hat{\lambda}_p(P)/\log p. 
\end{aligned}
\end{equation}
This is the same definition as in the algorithm 
but the value of $\Lambda$ is twice of that in the algorithm. 
Recall that in our definition 
the value of the canonical height is twice of that in \cite{sil1}.

For our elliptic curve, since $a_1=a_2=a_3=a_4=0$, 
$b_2=b_4=b_8=0$ and $b_6=4n$, we have 
\begin{equation} 
\label{eq:ABCL2} 
A=v_p\left(\frac{3\alpha^2}{\delta^4}\right),\ 
B=v_p\left(\frac{2\beta}{\delta^3}\right),\ 
C=v_p\left(\frac{3\alpha(\alpha^3+4n\delta^6)}{\delta^8}\right).
\end{equation}
Note that $c_4=0$ (i.e. $v_p(c_4)\neq 0$). 
This condition has an effect in the algorithm. 

On this condition, by the algorithm we have 
\begin{eqnarray}
\label{eq:Lambda}
\Lambda=
\left\{
\begin{array}{l}
2\max\left\{0, -\cfrac{1}{2}v_p(\alpha/\delta^2)\right\} 
\ \ \ {\rm if}\ A \leq 0\ {\rm or}\ B \leq 0 \\ \\ 
-\cfrac{2B}{3}\ \ \ {\rm if}\ A,\ B > 0,\ C\geq 3B  \\ \\
-\cfrac{C}{4}\ \ \ {\rm if}\ A,\ B > 0,\ C<3B 
\end{array}
\right. .
\end{eqnarray}

Now we consider the case of $p=2$. 
If $v_2(\alpha)=0$, then $A \leq 0$ and by (\ref{eq:Lambda}) 
\begin{equation*}
\hat{\lambda}_2(P)=\Lambda \log2
=2\max\left\{0, -\frac{1}{2}v_2(\alpha/\delta^2)\right\} \cdot \log2 
=2v_2(\delta) \log2.
\end{equation*}

We assume that $v_2(\alpha)\neq 0$. 
Then $v_2(\delta)=0$, since $\gcd(\alpha, \delta)=1$. 
So $A$, $B$ $> 0$. 
Since $P$ is on $E$, we have the equation $n\delta^6=\beta^2-\alpha^3$. 
Since $v_2(n)=0$, $v_2(\beta)=0$. 
So $B=v_2(2\beta)=1$ and 
$C=v_2(\alpha)+v_2(\alpha^3+4n\delta^6)\geq 1+2=3$. 
So $C\geq 3B$, and by (\ref{eq:Lambda}) 
\begin{equation*}
\hat{\lambda}_2(P)=\Lambda \log2=-\frac{2B}{3} \log2= -\frac{2}{3} \log2. 
\end{equation*}

\end{proof}
\begin{lem}
\label{lem:nonarc3}
We consider the situation of Lemma \ref{lem:nonarc2}.
Assume $v_3(n)=0$.
If $v_3(\beta)=0$, then $\hat{\lambda}_3(P)=2v_3(\delta)\log3$.
If $v_3(\beta)\neq 0$, then $\hat{\lambda}_3(P)=-\frac{1}{2}\log3$. 
\end{lem}

\begin{proof}
We compute $\hat{\lambda}_3(P)$ following (\ref{eq:ABCL2}), (\ref{eq:Lambda}) 
for $p=3$. 

If $v_3(\beta)=0$, then $B \leq 0$ and by (\ref{eq:Lambda}) 
\begin{equation*}
\hat{\lambda}_3(P)=\Lambda \log3
=2\max\left\{0, -\frac{1}{2}v_3(\alpha/\delta^2)\right\} \cdot \log3 
=2v_3(\delta) \log3.
\end{equation*}
The last equality is as follows.
If $v_3(\delta)=0$, then
$\max\left\{0, -\frac{1}{2}v_3(\alpha/\delta^2)\right\}=0$. 
So $\max\left\{0,
 -\frac{1}{2}v_3(\alpha/\delta^2)\right\}=v_3(\delta)$. 
If $v_3(\delta)\neq 0$, 
then since $\gcd(\alpha,\delta)=1, v_3(\alpha)=0$. 
So $\max\left\{0,
 -\frac{1}{2}v_3(\alpha/\delta^2)\right\}=v_3(\delta)$. 

We assume that $v_3(\beta)\neq 0$. 
Then $v_3(\delta)=0$, since $\gcd(\beta, \delta)=1$. 
So $B=v_3(2\be/\delta^3)=v_3(\be) > 0$ and 
$A=v_3(3\alpha^2/\delta^4)=v_3(3\alpha^2)>0$. 
Since $P$ is on $E$,  $n\delta^6=\beta^2-\alpha^3$. 
Since $v_3(n)=0$, $v_3(\alpha)=0$. 
Using the equality $\alpha^3+4n\delta^6=\beta^2+3n\delta^6$, 
\begin{equation*}
C=v_3(3\alpha)+v_3(\alpha^3+4n\delta^6) 
=v_3(3\alpha)+v_3(\beta^2+3n\delta^6) 
= 1+1=2. 
\end{equation*}
So we have $3B > C$. By (\ref{eq:Lambda}) 
\begin{equation*}
\hat{\lambda}_3(P)=\Lambda \log3=-\frac{C}{4} \log3= -\frac{1}{2} \log3. 
\end{equation*}

\end{proof}

\begin{lem}
\label{lem:nonarcp}
Let $n\in \Z$ be square-free and 
$E$ the elliptic curve $y^2=x^3+n$ over $\Q$.  
Let $P=(\alpha/\delta^2, \beta/\delta^3)$ 
$(\alpha, \beta, \delta \in \Z,\ \delta>0,\
\gcd(\alpha,\delta)=\gcd(\beta,\delta)=1)$ be 
a rational point on $E$.
We assume that $p\neq 2, 3$ 
. 
Then $\hat{\lambda}_p(P)=2v_p(\delta)\log p$.

\end{lem}

\begin{proof}
We compute $\hat{\lambda}_p(P)$ following (\ref{eq:ABCL2}), (\ref{eq:Lambda}). 
At first if $v_p(\alpha)=0$ or $v_p(\beta)=0$, 
then since $\delta$ is an integer,
$A\leq 0$ or $B \leq 0$. 
So 
\begin{equation*}
\hat{\lambda}_p(P)=\Lambda\log p
=2\max\left\{0, -\frac{1}{2}v_p(\alpha/\delta^2)\right\}\cdot \log p
=2v_p(\delta)\log p.
\end{equation*}
The last equality follows from the same reason as 
that in the proof of Lemma \ref{lem:nonarc3}. 

Next we assume that $v_p(\alpha) > 0$ and $v_p(\beta) > 0$. 
Then $v_p(\delta)=0$ because $\gcd(\alpha, \delta)=1$.
Since $v_p(\beta^2-\alpha^3)>1$ and 
$n\delta^6 = \beta^2-\alpha^3$, 
we have $v_p(n\delta^6) >1$. 
But $n$ is square-free, $v_p(n)=0$ or $1$. 
So this case does not happen.

\end{proof}

By the previous four lemmas, we have the following lemma.

\begin{lem}
\label{lem:finite part}
Let $n\in \Z$ be square-free and 
$E$ the elliptic curve $y^2=x^3+n$ over $\Q$.  
Let $P=(\alpha/\delta^2, \beta/\delta^3)$ 
$(\alpha, \beta, \delta \in \Z,\ \delta>0,\ 
\gcd(\alpha,\delta)=\gcd(\beta,\delta)=1)$ be 
a rational point on $E$.
Then the non-archimedean part of the canonical height of 
$P$ is as follows:
\begin{equation*}
\hat{h}_f(P)=  2 \log \delta+\lambda'_2(P)
+\lambda'_3(P), 
\end{equation*}
where 
\begin{eqnarray*}
\lambda'_2(P)=
\left\{
\begin{array}{r}
0\ \ \ \ \left( v_2(\alpha)=0 \right), \\
-\cfrac{2}{3}\log2\ \ \ \ \left( v_2(\alpha)\neq0 \right),
\end{array}
\right.
\end{eqnarray*}
\begin{eqnarray*}
\lambda'_3(P)=
\left\{
\begin{array}{r}
0\ \ \ \ \left( v_3(\beta)=0 \right), \\
-\cfrac{1}{2}\log3\ \ \ \ \left( v_3(\beta)\neq0 \right).
\end{array}
\right.
\end{eqnarray*}
\end{lem}
\begin{proof}
\begin{equation*}
\begin{aligned}
\hat{h}_f(P) & = \hat{\lambda}_2(P)+\hat{\lambda}_3(P) 
+\sum_{p\neq2, 3}\hat{\lambda}_p(P) \\
& = \hat{\lambda}_2(P)+\hat{\lambda}_3(P) + \sum_{p\neq2, 3}2v_p(\delta)\log p \\
& = \hat{\lambda}_2(P)-2v_2(\delta)\log 2 
+\hat{\lambda}_3(P)-2v_3(\delta)\log 3 
+2\log \prod_p p^{v_p(\del)} \\
& = \hat{\lambda}_2(P)-2v_2(\delta)\log 2 
+\hat{\lambda}_3(P)-2v_3(\delta)\log 3 
+2\log \delta. 
\end{aligned}
\end{equation*}
By Lemmas \ref{lem:nonarc2} and \ref{lem:nonarc3} 
we see that 
$\hat{\lambda}_2(P)-2v_2(\delta)\log 2$ and 
$\hat{\lambda}_3(P)-2v_3(\delta)\log 3$ are nothing but 
$\lambda'_2(P)$ and $\lambda'_3(P)$ respectively.

\end{proof}
\section{Uniform lower bound}
\label{sec:unif lower}
In this section we compute a uniform lower bound of the 
canonical height (Proposition \ref{prop:uniform lower}), 
that is a lower bound of the canonical height 
independent of $P \in E(\Q)$. 

\begin{prop}
\label{prop:cohen f}
Let $n\in \Z$ and 
let $E$ be the elliptic curve $y^2=x^3+n$ over $\Q$.  
Let $P=(\alpha/\delta^2, \beta/\delta^3)$ 
$(\alpha, \beta, \delta \in \Z,\ \delta>0,\ 
\gcd(\alpha,\delta)=\gcd(\beta,\delta)=1)$ be 
a rational point on $E$.
We assume that $n>0$. Then we have 
\begin{equation*}
\hat{\lambda}_{\infty}(P) > \frac{1}{12}\log n+\frac{1}{2}\log 
\left|\frac{\beta}{\delta^3}\right|
+0.31494685
.
\end{equation*}
\end{prop}

\begin{proof}
Recall that in our definition the value of the canonical height 
is twice of that in \cite{cohen0}.
By Algorithm 7.5.7 \cite{cohen0} and (\ref{eq:2y=}) 
\begin{equation}
\label{eq:cohen formula}
\hat{\lambda}_{\infty}(P)=\frac{1}{16}\log\left|\frac{\Delta}{q}\right| + 
\frac{1}{4}\log\left(\frac{\omega_1 y(P)^2}{2\pi}\right) 
- \frac{1}{2}\log\left|\theta\right| ,
\end{equation}
where 
$q=\exp(2\pi i{\omega_2}/{\omega_1})$, 
$\theta=\sum^{\infty}_{n=0}(-1)^nq^\frac{n(n+1)}{2}
\sin\left\{{2\pi}(2n+1) \re (z_P)/{\omega_1} \right\}$, 
$\Delta$ is the discriminant of $E$, 
$\omega_1$ and $\omega_2$ are periods of $E$ such that $\omega_1>0,\ 
\op{Im}(\omega_2)>0$ and $\op{Re}(\omega_2/\omega_1)=-1/2$ 
and $z_P$ is the elliptic logarithm of $P$. 
Recall that $z_P$ is the complex number 
in $\{t_1\omega_1+t_2\omega_2:0\leq t_1,t_2\leq 1\}$ 
such that $\wp(z_P)=x(P)$ and 
$\wp'(z_P)=2y(P)$, where $\wp$ is the Weierstrass $\wp$-function.

Note that $q$ is a real number
since 
\begin{equation*}
\begin{aligned}
q=\exp\left(2\pi i\frac{\omega_2}{\omega_1}\right)
& = \exp\left(2\pi i \left(-\frac{1}{2}+i\im\left(\frac{\omega_2}{\omega_1}
\right)\right)\right) \\
& = \exp\left(-\pi i -2\pi\im\left(\frac{\omega_2}{\omega_1}\right)\right)
= -\exp\left(-2\pi \im\left(\frac{\omega_2}{\omega_1}\right)\right). 
\end{aligned}
\end{equation*}

By Definition 7.4.6 and Algorithm 7.4.7 in \cite{cohen0} 
\begin{equation*}
\begin{aligned}
\omega_1 
& = \frac{2\pi}{\op{AGM}(2\sqrt[4]{3}n^{\frac{1}{6}}, 
\sqrt{2\sqrt{3}-3}n^{\frac{1}{6}})}
= n^{-\frac{1}{6}} \cdot 
\frac{2\pi}{\op{AGM}(2\sqrt[4]{3}, \sqrt{2\sqrt{3}-3})}, 
\end{aligned}
\end{equation*}
where $\op{AGM}(\cdot, \cdot)$ is the arithmetic geometric mean. 
So if we let $\omega_1'$, $\omega_2'$ be the periods 
of the elliptic $y^2=x^3+1$, 
then we have 
$\omega_1=n^{-\frac{1}{6}}\times$ $\omega_1'$. 
It turns out that 
$\omega_1'=4.206546315 \cdots$.  This can be 
done by PARI/GP (Version 2.3.4) (\cite{pari}) as follows. 
\\ \\
\verb|E1=ellinit([0,0,0,0,1]);|\\
\verb|E1.omega|\\

Similarly by \cite[Algorithm 7.4.7]{cohen0}, we have 
\begin{equation*}
\begin{aligned}
\omega_2/\omega_1 & = -\frac{1}{2}+\frac{i}{2}
\frac
{\op{AGM}(2\sqrt[4]{3}n^{\frac{1}{6}}, \sqrt{2\sqrt{3}+3}n^{\frac{1}{6}})}
{\op{AGM}(2\sqrt[4]{3}n^{\frac{1}{6}}, \sqrt{2\sqrt{3}-3}n^{\frac{1}{6}})}\\
& = -\frac{1}{2}+\frac{i}{2}
\frac
{\op{AGM}(2\sqrt[4]{3}, \sqrt{2\sqrt{3}+3})}
{\op{AGM}(2\sqrt[4]{3}, \sqrt{2\sqrt{3}-3})}
=\omega_2'/\omega_1' 
\end{aligned}
\end{equation*}
and so it turns out that 
$q=-0.163033534 \cdots$ by PARI/GP
as follows(the above commands are needed).
\\ \\
\verb|-exp(-2*Pi*imag(E1.omega[2]/E1.omega[1]))|\\ 

Substituting these values 
and $\Delta=-432n^2$ in (\ref{eq:cohen formula}), 
we have 
\begin{equation*}
\begin{aligned}
\hat{\lambda}_{\infty}(P)
&=\frac{1}{16}\log\left|\frac{432n^2}{q}\right| + 
\frac{1}{4}\log\left(\frac{n^{-\frac{1}{6}}\omega'_1\beta^2}
{2\pi\delta^6}\right) 
- \frac{1}{2}\log\left|\theta\right| \\
&>\frac{1}{16}\log\left|\frac{432n^2}{0.163033535}\right| + 
\frac{1}{4}\log\left(\frac{4.206546315 n^{-\frac{1}{6}}\beta^2}
{2\pi\delta^6}\right) 
- \frac{1}{2}\log\left|1.167385748\right| \\
&=\frac{1}{12}\log n+\frac{1}{2}\log 
\left|\frac{\beta}{\delta^3}\right|
+0.3149468597 \cdots  
\end{aligned}
\end{equation*}
by the trivial bound $|\theta|
< 1+|q|+|q|^3+|q|^6+|q|^{10}+|q|^{15}+|q|^{21}+ \cdots
< 1+|q|+|q|^3+|q|^6+\frac{|q|^{10}}{1-|q|^5}$=$1.16738574713\cdots$.  

\end{proof}

\begin{prop}
\label{prop:uniform lower}
Let $n$ be a positive, square-free integer and  
$E$ the elliptic curve $y^2=x^3+n$. 
If $P$ is a rational, non-torsion point on $E$,
then 
\begin{equation}
\label{eq:uniform lower}
\hat{h}(P) > \frac{1}{12}\log n
-0.147152
.
\end{equation}
\end{prop}

\begin{proof}
By Lemmas \ref{lem:nonarc2}, \ref{lem:nonarc3}, 
\ref{lem:finite part} and Proposition \ref{prop:cohen f}, we have 
\begin{equation*}
\hat{h}(P) =\hat{h}_f(P) + \hat{\lambda}_{\infty}(P)
\end{equation*}
\begin{equation*}
> 2 \log \delta +\lambda'_2(P)
+\lambda'_3(P)
+\frac{1}{12}\log n+\frac{1}{2}\log\left|\frac{\beta}{\delta^3}\right|
+0.31494685 
\end{equation*}
\begin{equation*}
= \frac{1}{2}\log \delta+\lambda'_2(P)
+\left\{ \lambda'_3(P)
+\frac{1}{2} \log|\beta| \right\}
+\frac{1}{12}\log n
+0.31494685 
\end{equation*}
\begin{equation*}
\geq 
\frac{1}{12}\log n
-\frac{2}{3}\log 2 
+0.31494685 
=\frac{1}{12}\log n
-0.1471512 \cdots, 
\end{equation*}
since $\delta \in \Z$ and 
$\lambda'_3(P)+\frac{1}{2} \log|\beta| \geq 0$.
\end{proof}

\section{Estimate of the lattice index}
\label{sec:lattice index}
Let E be an elliptic curve of rank $r (\geq 2)$ 
defined over a number filed $K$. 
Let $Q_1, Q_2, ..., Q_s$ $(s\leq r)$ 
be independent points in $E(K)$. 
Then there exist generators $G_1, G_2, ..., G_r$ of 
the free part of $E(K)$ 
such that $Q_1, Q_2, ..., Q_s \in \Z G_1+\Z G_2+ \cdots + \Z G_s $ 
by the elementary divisor theory. 
The index of the subgroup $\Z Q_1+ \Z Q_2+ \cdots + \Z Q_s $ in 
$\Z G_1+ \Z G_2+ \cdots+ \Z G_s$ is called 
the {\emph {lattice index}} of $\{Q_1, Q_2, ..., Q_s\}$. 
We put 
\begin{align*}
& \langle Q_i, Q_j \rangle
=\frac{1}{2}\left( \hat{h}(Q_i+Q_j)-\hat{h}(Q_i)-\hat{h}(Q_j) \right), \\
& {R}(Q_1, Q_2, ..., Q_s)
=\det \left(\langle Q_i, Q_j \rangle\right)_{1\leq i, j \leq s} 
.\end{align*}
It is known that the canonical height $\hat{h}$ is a positive 
definite quadratic form on $E(K)/E(K)_{\rm tors}$. 
When we identify $E(K)/E(K)_{\rm tors} 
\simeq \Z G_1+ \Z G_2+ \cdots + \Z G_r$ as $\Z$-modules,  
$\hat{h}$ is the quadratic form defined by 
the symmetric matrix $(\langle G_i, G_j \rangle)_{1\leq i, j \leq r}$.

Let $f({\bf x})=\sum_{i,j=1}^n f_{i,j}x_ix_j$ be a positive definite 
symmetric quadratic form. 
Then it is known that there exists a constant $\gam_n$ 
called the {\emph{Hermite constant}}  
such that 
\begin{equation*}
\inf_{{\bf m} \in \Z^r\setminus\{0\}}f({\bf m})\leq \gam_n \det(f_{i,j}). 
\end{equation*}
For example, 
\begin{align*}
\gamma_1^1&=1,\ \gamma_2^2=4/3,\ \gamma_3^3=2,\ \gamma_4^4=4,\ \ldots
\end{align*}

In this section we estimate the lattice index. 
For this we use the following theorem of Siksek.

\begin{thm}
{\upshape (\cite[Theorem 3.1]{siksek1})}
\label{thm:siksek}
Let $E$ be an elliptic curve of rank $r\ (\geq 2)$ 
defined over a number field $K$. 
Let $Q_1, Q_2, ..., Q_s$ $(s\leq r)$ 
be independent points in $E(K)$ and $\nu$ the lattice index 
of $\{Q_1, Q_2, ..., Q_s\}$. 
Suppose that $\lambda>0$ is a constant such that 
any point $P \in E(K)$ of infinite order  
satisfies $\hat{h}(P)> \lambda$. 
Then 
\begin{equation*}
\nu \leq {R}(Q_1, Q_2, ..., Q_s)^{1/2} (\gamma_s/\lambda)^{s/2}. 
\end{equation*}

\end{thm}
\begin{prop}
\label{less5}
Assume that 
$m=a^6+16b^6$ is square-free, $ab$ is odd and 
the discrete valuation $v_3(b)$ equals $1$. 
If $m > 6.38 \times 10^{22}$   
$($this is true for either $a>6321$ or $b>3982$$)$, 
the lattice indices of $\{P_1, P_2\}$, $\{P_2, P_3\}$, $\{P_3, P_1\}$ 
are less than $5$. 
If $m > 19088 $ 
$($this is always true$)$, 
the lattice indices of $\{P_1, P_2\}$, $\{P_2, P_3\}$, $\{P_3, P_1\}$ 
are less than $7$. 
\end{prop}

\begin{proof}
In this situation $P_1, P_2, P_3$ are independent 
by Proposition \ref{indep} in the next section. 
Let $\lambda=\frac{1}{12}\log m -0.147152$. Then
$\hat{h}(P)>\lambda$ for any 
non-torsion point $P \in E_{a,b}(\Q)$. 
Now by Theorem \ref{thm:siksek}, it suffices to show 
that 
\begin{math}
{R}(P_i, P_j)^{1/2} (\gamma_2/\lambda)^{2/2} 
\end{math}
is less than $5$ or $7$, when $m > 6.38 \times 10^{22}$ or $m > 19088$ 
respectively 
for $i\not=j$ $(i,j=1, 2, 3)$.
Since 
\begin{equation*}
{R}(P_2, P_3)
=\hat{h}(P_2)\hat{h}(P_3)
-\frac{1}{4}\left\{\hat{h}(P_2+P_3)-\hat{h}(P_2)-\hat{h}(P_3)\right\}^2,   
\end{equation*}
we have 
\begin{equation*}
\begin{aligned}
\label{eq:ineq}
\left\{{R}(P_2, P_3)^{1/2} (\gamma_2/\lambda)^{2/2} \right\}^2
&=\frac{4}{3}
\frac{\hat{h}(P_2)\hat{h}(P_3)
-\frac{1}{4}\left\{\hat{h}(P_2+P_3)-\hat{h}(P_2)-\hat{h}(P_3)\right\}^2}
{\lambda^2} \\
&<\frac{4}{3}
\frac{\hat{h}(P_2) \hat{h}(P_3)}{\lambda^2} \\
&<\frac{4}{3}
\frac{(\frac{1}{3} \log m +1.0515)(\frac{1}{3} \log m +0.5665)}
{(\frac{1}{12}\log m -0.147152)^2} . 
\end{aligned}
\end{equation*}
The last inequality follows from  
Propositions \ref{prop1} and \ref{prop:uniform lower}. 
By elementary calculus we see that 
the last bound 
is less than 25 
if $m > 6.38 \times 10^{22}$, 
less than 49 if $m>19088$ 
and decreasing if $m > e^2$. 

Since the upper bound of $\hat{h}(P_1)$ 
given in Proposition \ref{prop1} 
is less than 
those of $\hat{h}(P_2)$ and 
$\hat{h}(P_3)$, the cases of $\{P_1, P_2\}$, $\{P_3, P_1\}$ 
are clear. 
\end{proof}

\section{Independence of $P_1, P_2, P_3$}
\label{sec:indep}
In this section we show that 
in the situation of Proposition \ref{less5},  
$P_1$, $P_2$, $P_3$ are independent and 
the lattice index of 
$\{P_i,P_j\}$ $(i\neq j)$ is not divisible by $2, 3$. 
\begin{lem}
\label{lem2d}
Let $n\in\Z$ and let $E$ be the elliptic curve $y^2=x^3+n$ over $\Q$ 
and $Q \in E(\Q)\setminus E(\Q)_{\op{tors}}$. 
We write $x(Q)=u/s^2$ with $\gcd(u,s)=1$.
Then 
$Q \not\in 2E(\Q)$ in either of the following cases:
\begin{itemize}
\upshape
\item[(1)] $n$ is odd, 
$u \not\equiv0 \pmod 8$ and $s$ is odd, 
\item[(2)] $n\equiv 1 \pmod 9$, $u\equiv2 \pmod 3$ and 
$s\not\equiv 0 \pmod3$.
\end{itemize}
\end{lem}

\begin{proof}
We assume that there exists $R=(w/t^2, z/t^3) \in E(\Q)$ 
with $\gcd(w,t)=1$ 
such that $Q=2R$ 
and deduce a contradiction. 
By (\ref{eq:map2}) or the following PARI/GP commands, 
\begin{verbatim}
En=ellinit([0,0,0,0,n]);
ellpow(En,[w/t^2,z/t^3],2)[1]
\end{verbatim}
we have 
$x(2R)=(9w^4-8wz^2)/(4z^2t^2)$ and so 
$u/s^2=(9w^4-8wz^2)/(4z^2t^2)$. 
On the other hand $(z/t^3)^2=(w/t^2)^3+n$ since $R$ is on $E$. 
Eliminating $z$, 
\begin{equation}
\label{eq:2-des}
s^2w(w^3-8nt^6)=4ut^2(w^3+nt^6).
\end{equation}

(1) If $n$ and $s$ are odd, 
then $w$ is even by (\ref{eq:2-des}). 
Further $t$ is odd since $\gcd(w, t)=1$. 
Then $v_2(w(w^3-8nt^6))\geq 5$ (note that if $v_2(w)=1$, 
$w^3-8nt^6=8\times$ even). 
So $v_2(4ut^2(w^3+nt^6)) \geq 5$ 
and therefore $v_2(u) \geq 3$.  
This is a contradiction since $u \not\equiv0 \pmod8$. 

(2) Assume that 
$n\equiv 1 \pmod 9$, $u\equiv2 \pmod 3$ and $s\not\equiv 0 \pmod3$. 
Note that if $x \not\equiv 0 \pmod3$, 
then $x^2 \equiv 1 \pmod 9$ (so modulo $3$ also). 
 
Assume $w \equiv 0 \pmod3$. 
Then $t \not\equiv 0 \pmod3$ 
since $\gcd(w,t)=1$. 
So the left hand side of (\ref{eq:2-des}) 
$\equiv 0 \pmod3$ and the right hand side of (\ref{eq:2-des}) 
$\not\equiv 0 \pmod3$. This is a contradiction. 

Assume $w \equiv 1 \pmod3$. 
If $t \equiv 0 \pmod3$, then 
the left hand side of (\ref{eq:2-des}) 
$\equiv 1 \pmod3$ and the right hand side of (\ref{eq:2-des}) 
$\equiv 0 \pmod3$. This is a contradiction. 

If $t \not\equiv 0 \pmod3$, then 
the left hand side of (\ref{eq:2-des}) 
$\equiv 2 \pmod3$ and the right hand side of (\ref{eq:2-des}) 
$\equiv 1 \pmod3$. This is a contradiction. 

Assume $w \equiv -1 \pmod3$. 
If $t \equiv 0 \pmod3$, then 
the left hand side of (\ref{eq:2-des}) 
$\not\equiv 0 \pmod3$ and the right hand side of (\ref{eq:2-des}) 
$\equiv 0 \pmod3$. This is a contradiction. 

Note that $w^3 \equiv -1 \pmod9$. 

If $t \not\equiv 0 \pmod3$, 
then $w^3-8nt^6 \equiv 0 \pmod9$ 
and $w^3+nt^6 \equiv 0 \pmod9$. 
So we can write $w^3-8nt^6=9W_1$, $w^3+nt^6=9W_2$. 
Then by (\ref{eq:2-des}) we have 
$s^2w \cdot9W_1 \equiv 4ut^2\cdot 9W_2 \pmod{27}$. 
So $s^2wW_1 \equiv 4ut^2 W_2 \pmod3$. 
Therefore $-W_1 \equiv -W_2 \pmod3$. 
On the other hand $9W_2-9W_1=9nt^6$ and so 
$W_2-W_1=nt^6 \not\equiv 0 \pmod3$. 
This is a contradiction. 
\end{proof}

\begin{rem}
\label{rem:2-des}
\upshape
Assume that we can write $x(Q)=u/s^2=u'/s'^2$ 
($u',s' \in\Z$ and not necessarily $\gcd(u',s')=1$). 
So $u|u'$ and $s|s'$. 
Then if $u' \not\equiv 0\pmod8$, $u \not\equiv 0\pmod8$. 
If $s'$ is odd, $s$ is odd. 
If $s' \not\equiv 0\pmod3$, $s \not\equiv 0\pmod3$. 
If $u' \equiv 2\pmod3$, $u \equiv 2\pmod3$ 
since $u'=(s'/s)^2u$ and $s'/s \not\equiv 0\pmod3$. 

So it is not necessary to assume $\gcd(u,s)=1$ in 
Lemma \ref{lem2d}. 
\end{rem}
\begin{lem}
\label{lem3d}
Let $n\in\Z$ and let $E$ be the elliptic curve $y^2=x^3+n$ over $\Q$ 
and $Q \in E(\Q)\setminus E(\Q)_{\op{tors}}$. 
We write $x(Q)=u/s^2$ with $\gcd(u,s)=1$.
Then $Q \not\in 3E(\Q)$ in either of the following cases: 
\begin{itemize}
\upshape
\item[(1)] $n$ is odd and $u$ is even, 
\item[(2)] $n\equiv1 \pmod 9$, $u\equiv1 \pmod 3$
and  $v_3(s)=1$. 
\end{itemize}
\end{lem}
\begin{proof}
We assume that there exists $R=(w/t^2, z/t^3) \in E(\Q)$ with $\gcd(w,t)=1$ 
such that $Q=3R$ 
and deduce a contradiction. 
By the following PARI/GP commands 
\begin{verbatim}
En=ellinit([0,0,0,0,n]);
ellpow(En,[w/t^2,z/t^3],3)[1]
\end{verbatim}
we have 
$x(3R)=(64z^6-144w^3z^4+81w^9)/9t^2w^2(4z^2-3w^3)^2$ 
and so $u/s^2=(64z^6-144w^3z^4+81w^9)/9t^2w^2(4z^2-3w^3)^2$. 
On the other hand $(z/t^3)^2=(w/t^2)^3+n$ since $R$ is on $E$. 
Eliminating $z$, 
\begin{equation}
\label{eq:3-des}
s^2\left\{(w^3+4nt^6)^3-2^23^3nw^6t^6\right\}=3^2uw^2t^2(w^3+4nt^6)^2.
\end{equation}

(1) If $u$ is even, then $s$ is odd since $\gcd(u,s)=1$. 
Then since $(w^3+4nt^6)^3-2^23^3nw^6t^6$ is even, 
$w$ must be even. So $t$ is odd since $\gcd(w,t)=1$.  
Since $n$ is odd, 
$v_2(w^3+4nt^6)=2$ and therefore 
$v_2({\rm the\ left\ hand\ side\ of\ } (\ref{eq:3-des}))=6$. 
On the other hand 
$v_2({\rm the\ right\ hand\ side\ of\ } (\ref{eq:3-des})) \geq 7$. 

(2) If $v_3(s)=1$, we can write $s=3s'\ (s'\not\equiv0 \pmod3)$. 
So we have 
\begin{equation}
\label{eq:3-des2}
s'^2\left\{(w^3+4nt^6)^3-2^23^3nw^6t^6\right\}=uw^2t^2(w^3+4nt^6)^2.
\end{equation}

Now we show $wt\not\equiv 0 \pmod3$. 
Assume that $wt\equiv 0 \pmod3$. 
Then since the each side of (\ref{eq:3-des2}) $\equiv 0 \pmod3$, 
we have $(w^3+4nt^6)^3-2^23^3nw^6t^6\equiv 0\pmod3$. 
So $w^3+4nt^6 \equiv 0 \pmod3$. 
But this does not happen since $\gcd(w,t)=1$ and $n\equiv1 \pmod 9$. 
So we see $wt\not\equiv 0 \pmod3$.

Now if we assume that $w\equiv -1 \pmod3$, then 
$w^3+4nt^6 \equiv -1+4t^6 \equiv 3 \pmod9$.  
So $v_3(w^3+4nt^6)=1$. 
Then 
$v_3({\rm the\ left\ hand\ side\ of\ } (\ref{eq:3-des2}))\geq 3$ \\
and 
$v_3({\rm the\ right\ hand\ side\ of\ } (\ref{eq:3-des2}))=2$. 
This is a contradiction. 

If we assume that $w\equiv 1 \pmod3$, then $w^3+4nt^6 
\equiv -1 \pmod3$. Then seeing (\ref{eq:3-des2}) modulo 3, 
we have $u\equiv -1 \pmod3$. This is a contradiction.

\end{proof}

\begin{prop}
\label{indep}
We assume that 
$m=a^6+16b^6$ is square-free, $ab$ is odd and 
the discrete valuation $v_3(b)$ equals $1$. 
Then 
$P_1$, $P_2$, $P_3$, $P_1+ P_2$, $P_2+ P_3$, $P_1+ P_3$, 
$P_1+P_2+P_3$ $\not\in 2E_{a,b}(\Q)$ 
and 
$P_1$, $P_2$, $P_3$, $P_1\pm P_2$, 
$P_2\pm P_3$, $P_1\pm P_3$, $P_1+P_2\pm P_3$, 
$P_1-P_2\pm P_3$ $\not\in 3E_{a,b}(\Q)$. 
In particular, $P_1$, $P_2$, $P_3$ are independent 
and the lattice indices of 
$\{P_1,P_2,P_3\}$, $\{P_1,P_2\}$,  $\{P_2,P_3\}$, $\{P_3,P_1\}$ 
are not divisible by $2$ nor $3$. 
\end{prop}
\begin{proof}
To ease the notation, we put $E=E_{a,b}$. 
We have 
\begin{equation*}
x(P_1)=-a^2,\ x(P_2)=2ab,\ x(P_3)=-2ab,
\end{equation*}
\begin{align*}
x(P_1+P_2) & =\frac{2a(a^3+a^2b-2ab^2-4b^3)}{(a+2b)^2}, \\
x(P_1-P_2) & =\frac{2(a^4-3a^3b+6a^2b^2-8ab^2+8b^4)}{a^2}, \\
x(P_1+P_3) & =\frac{2(a^4+3a^3b+6a^2b^2+8ab^3+8b^4)}{a^2}, \\
x(P_1-P_3) & =\frac{2a(a^3-a^2b-2ab^2+4b^3)}{(a-2b)^2}, \\
x(P_2+P_3) & =\frac{4b^4}{a^2}, \
x(P_2-P_3)=\frac{a^4}{(2b)^2}, 
\end{align*}
\begin{equation*}
x(P_1+P_2+P_3)=\frac{2a(a^5+4a^4b+8a^3b^2+12a^2b^3+14ab^4+8b^5)}
{(a^2+2ab+2b^2)^2},\ 
\end{equation*}
\begin{equation*}
x(P_1-P_2-P_3)=\frac{2a(a^5-4a^4b+8a^3b^2-12a^2b^3+14ab^4-8b^5)}
{(a^2-2ab+2b^2)^2}.\ 
\end{equation*}
Note that $m=a^6+16b^6 \equiv 1\pmod9$ since $v_3(b)=1$ and $\gcd(a,b)=1$. 
As we saw in Remark \ref{rem:2-des}, 
we can use Lemma \ref{lem2d} without the assumption 
that the $x$-coordinate is an irreducible fraction. 
Note that $m$ in this proposition corresponds to $n$ 
in Lemma \ref{lem2d}.  

We see that 
$P_1+P_2 \not\in 2E(\Q)$ by Lemma \ref{lem2d}$(2)$ 
since $2a(a^3+a^2b-2ab^2-4b^3) \equiv 2a^4 \equiv 2\pmod3$ and 
$a+2b \equiv a \not\equiv 0\pmod3$.  
Similarly 
$P_1+P_3$ $\not\in 2E(\Q)$ 
by Lemma \ref{lem2d}$(2)$. 
It is clear that 
$P_1$, $P_2$, $P_3$, $P_2+ P_3$, $P_1+P_2+P_3$ 
$\not\in 2E(\Q)$ by Lemma \ref{lem2d}$(1)$. 

If there is a rational point $R$ such that 
$P_1=3R$, then $\hat{h}(P_1)=9\hat{h}(R)$. 
But by Proposition \ref{eq:uniform lower} we have 
$9\hat{h}(R) > 9(\frac{1}{12}\log m -0.147152) 
> 
\frac{1}{3} \log m +0.5409>\hat{h}(P_1)$ for $m\geq 88$, 
which is a contradiction.  So 
$P_1\notin 3E(\Q)$.

Since $a^4/(2b)^2$ is an irreducible fraction,   
by Lemma \ref{lem3d}$(2)$ 
$P_2-P_3\notin 3E(\Q)$. 
By computations we have \\
\\
$x(2P_1-2P_2-P_3)$\\
$
=a (-6144 b^{17}+34816 a b^{16}-101376 a^{2} b^{15}+204544 a^{3} b^{14}-320128 a^{4} b^{13}+409472 a^{5} b^{12}-
439840 a^{6} b^{11}+403168 a^{7} b^{10}-318248 a^{8} b^{9}+217216 a^{9} b^{8}-128160 a^{10} b^{7}+65072 a^{11} b^{6}-28152 a^{12} b^{5}+
10200 a^{13} b^{4}-3006 a^{14} b^{3}+684 a^{15} b^{2}-108 a^{16} b+9 a^{17})/b^{2} (2 b-a)^{2} 
(16 b^{6}-40 a b^{5}+56 a^{2} b^{4}-46 a^{3} b^{3}+28 a^{4} b^{2}-12 a^{5} b+3 a^{6})^{2}.
$ \\ 
\\
We denote the numerator by $U'$ 
and the denominator by $S'^2$.
Further we write $U'/S'^2=U/S^2$ as 
an irreducible fraction since it is an $x$-coordinate
of an elliptic curve.  
Since $v_3(9a^{17})=2$ and the orders of other terms of 
$U'$ is greater than $2$, $v_3(U')=2$.   
In $S'$, $v_3(b^2)=2,v_3(3a^6)=1$ and other factors are
not divisible by $3$. So $v_3(S'^2)=4$. 
Therefore, $v_3(S)=1$ and $U'':=U'/9$, $S'':=S'/9$ are integers. 
Clearly $U''/S''^2=U/S^2$. 
Since $U'' \equiv a^{18} \equiv 1\pmod3$, $U \equiv 1$ 
by the same argument as in Remark \ref{rem:2-des}. 
So $2P_1-2P_2-P_3 \not\in 3E(\Q)$ by Lemma \ref{lem3d}$(2)$. 
Therefore $P_1-P_2+P_3$=$-(2P_1-2P_2-P_3)+3(P_1-P_2) \not\in 3E(\Q)$. 
We have \\
\\
$x(2P_1+2P_2+P_3)$\\
$
=(4096 b^{18}+24576 a b^{17}+71680 a^{2} b^{16}+135680 a^{3} b^{15}+188160 a^{4} b^{14}+204800 a^{5} b^{13}+181632 
a^{6} b^{12}+133536 a^{7} b^{11}+83488 a^{8} b^{10}+48472 a^{9} b^{9}+30720 a^{10} b^{8}+22464 a^{11} b^{7}+16496 a^{12} b^{6}+10584 
a^{13} b^{5}+5496 a^{14} b^{4}+2178 a^{15} b^{3}+612 a^{16} b^{2}+108 a^{17} b+9 a^{18})/a^{2} b^{2} 
(48 b^{6}+128 a b^{5}+156 a^{2} b^{4}+114 a^{3} b^{3}+56 a^{4} b^{2}+18 a^{5} b+3 a^{6})^{2}
$\\
\\
and by the same argument as above, 
we have $2P_1+2P_2+P_3 \not\in 3E(\Q)$ by Lemma \ref{lem3d}$(2)$. 
Therefore $P_1+P_2-P_3$=$-(2P_1+2P_2+P_3)+3(P_1+P_2) \not\in 3E(\Q)$. 
We see that 
$P_2$, $P_3$, $P_1\pm P_2$, 
$P_2 + P_3$, $P_1\pm P_3$, $P_1+P_2 + P_3$, 
$P_1-P_2 - P_3$ $\not\in 3E(\Q)$ 
by Lemma \ref{lem3d}$(1)$, since the denominators 
of the $x$-coordinates of them are all odd. 

Next we prove the latter assertion of the proposition. 
By the elementary divisor theory 
there are generators $G_1,\ldots,G_r \in E(\Q)$ 
and $M\in M_3(\Z)$ 
such that 
\begin{eqnarray}
\left[ 
\begin{array}{c}
P_1 \\
P_2 \\
P_3 \\
\end{array} 
\right]
=M
\left[ 
\begin{array}{c}
G_1 \\
G_2 \\
G_3 \\
\end{array} 
\right]
. 
\end{eqnarray}

Note that the lattice index of 
$\{P_1,P_2,P_3\}$ equals $|\det M|$. 
Let $p$ be a rational prime. 
We have 
\begin{eqnarray}
\label{eq:M modp}
\left[ 
\begin{array}{c}
P_1 \\
P_2 \\
P_3 \\
\end{array} 
\right]
\equiv \bar{M}
\left[ 
\begin{array}{c}
G_1 \\
G_2 \\
G_3 \\
\end{array} 
\right] \pmod {pE(\Q)}, 
\end{eqnarray}
where $\bar{M}$ is the image of $M$ in $M_3(\Z/p\Z)$. 
We assume that there exists $A\in \gl_3(\Z/p\Z)$ such that 
$A\bar{M}$ has the row $[\bar{0}\ \bar{0}\ \bar{0}]$ and 
deduce a contradiction. 
Since we may assume that the first row is 
$[\bar{0}\ \bar{0}\ \bar{0}]$, by the left multiplication 
of $A$ on (\ref{eq:M modp}) we have 
\begin{eqnarray}
\left[ 
\begin{array}{c}
k_1P_1 +k_2P_2 +k_3P_3 \\
* \\
* \\
\end{array} 
\right]
\equiv 
\left[ 
\begin{array}{ccc}
\bar{0} & \bar{0} & \bar{0} \\
* & * & * \\
* & * & * \\
\end{array} 
\right]
\left[ 
\begin{array}{c}
G_1 \\
G_2 \\
G_3 \\
\end{array} 
\right] \pmod {pE(\Q)}, 
\end{eqnarray}
where $[k_1\ k_2\ k_3]$ is the first row of $A$. 
But the former assertion of this proposition 
implies that $k_1P_1 +k_2P_2 +k_3P_3 \notin pE(\Q)$ $(p=2,3)$ 
for any $(k_1,k_2,k_3) \in \left(\Z/p\Z\right)^3\setminus(\bar{0},\bar{0},\bar{0})$. 
This is a contradiction. 
Therefore $\det M$ is not congruent to 
$0$ modulo 2 or modulo 3. 

By the same argument as above,  
the cases of 
$\{P_1,P_2\}$, $\{P_2,P_3\}$, $\{P_3,P_1\}$ 
follow.
\end{proof}

\begin{rem}
\label{rem:5-des}
By the same reason as above, 
if we verify that 
$P_1$, $P_2$, $P_3$, 
$P_1 \pm P_2$, $P_2 \pm P_3$, $P_3 \pm P_1$, 
$P_1 \pm 2P_2$, $P_2 \pm 2P_3$, $P_3 \pm 2P_1$ 
$\notin 5E_{a,b}(\Q)$, 
we can prove that 
the lattice indices of 
$\{P_1,P_2\}$, $\{P_2,P_3\}$, $\{P_3,P_1\}$ 
are not divisible by $5$. 
Note that $P \notin 5E(\Q)$ amounts to $kP \notin 5E(\Q)$ $(k=\pm1,\pm2)$. 
For $3 \leq a \leq 6321$, $3 \leq b \leq 3982$ 
we can verify this by the software Magma $($\cite{magma}$)$. 
We give the code for this in Appendix \ref{sec:magma code}. 
\end{rem}

Now we can finish the proof of our main theorem. 

\begin{proof}
[Proof of Theorem \ref{main}]
For $a>6321, b>3982$ 
by Propositions \ref{less5}, \ref{indep} 
the lattice indices of 
$\{P_1,P_2\}$, $\{P_2,P_3\}$, $\{P_3,P_1\}$ 
equal 1.
For $5 \leq a \leq 6321, 3 \leq b \leq 3982$ 
by Propositions \ref{less5}, \ref{indep} and Remark \ref{rem:5-des} 
the lattice indices of 
$\{P_1,P_2\}$, $\{P_2,P_3\}$, $\{P_3,P_1\}$ 
equal 1. This completes the proof of Theorem \ref{main}.

\end{proof}

We prove that there are infinitely many 
$(a,b)$'s which satisfy the condition  
of Theorem \ref{main}. 
\begin{lem}
The set 
\begin{equation*}
S:=\left\{m=a^6+16b^6\in \Z \suchh
\begin{array}{l}
a, b \in \Z,\ m: {\rm square\textrm{-}free}\\
v_2(ab)=0,\ v_3(b)=1 
\end{array}\right\} 
\end{equation*}
is an infinite set. 
\end{lem}

\begin{proof}

We put 
\begin{equation*}S_{0}:=
\left\{m=(2k+3l)^6+16(6k-9l)^6\in \Z \suchh
\begin{array}{l}
k,l \in \Z,\ m: {\rm square\textrm{-}free}
\end{array}\right\}. 
\end{equation*}
For $(2k+3l)^6+16(6k-9l)^6$ being square-free 
it is necessary that $v_3(k)=v_2(l)=0$. 
Hence $S_0$ is a subset of $S$. 
From the Theorem of Greaves 
(\cite[THEOREM]{greaves1})
we see that $S_0$ is an infinite set, since 
$
(2x+3y)^6+16(6x-9y)^6
=8503785 y^6-34009308 x y^5+56691900 x^2 y^4-50384160 x^3 y^3
+25196400 x^4 y^2-6717888 x^5 y+746560 x^6
$ 
is an irreducible polynomial over $\Z$. This is verified by 
the command \verb|factor| of the software Maple (\cite{maple}). 
Therefore $S$ is an infinite set. 

\end{proof}

\section{Uniform bounds of $z'(P')$}
\label{sec:uni z'}
We use the notation of (\ref{eq:E'}), (\ref{eq:z'}). 
In this section we prove Lemmas \ref{remboundu} 
and \ref{remboundl}, which were
used in Proposition \ref{prop1}
to give bounds of $z'(P')$ 
independent of $P \in E_{a,b}(\Q)$.

Although the following computations can be done by manually 
except the numerical evaluations, 
we included Maple commands to carry out 
all the steps at end of this section.

\begin{proof}
[Proof of Lemma \ref{remboundu}]
Let $x=x(P)$. 
In this case since $E'$ is $y^2=(x-d)^3+m $, we have 
$a'_1=a'_3=0$, $a'_2=-3d$, $a'_4=3d^2$, $a'_6=m-d^3$, 
$b'_4=6d^2,\ b'_6=4m-4d^3,\ b'_8=3d^4-12dm$. 
By (\ref{eq:z'}) 
\begin{equation*}
z'(P')
=1-\frac{6d^2}{(x+d)^2}-2\frac{4m-4d^3}{(x+d)^3}
-\frac{3d^4-12dm}{(x+d)^4}
=\frac{x^4+4dx^3-8mx+4dm}{(x+d)^4}. 
\end{equation*}
Since $x^3+m=y^2\geq 0$, $x\geq -m^{1/3}$. 
Note that $d>m^{1/3}$, 
since $(2a^2+4b^2)^3-(a^6+16b^6)=
48b^6+96a^2b^4+48a^4b^2+7a^6 > 0$. 

If $x\geq 0$
\begin{equation*}
\begin{aligned}
\frac{x^4+4dx^3-8mx+4dm}{(x+d)^4} 
&\leq 
\frac{x^4}{(x+d)^4}
+\frac{4dx^3}{(x+d)^4}+\frac{4dm}{(x+d)^4} \\
&< 1+4+4=9. 
\end{aligned}
\end{equation*}

If $x< 0$
\begin{equation*}
\begin{aligned}
\frac{x^4+4dx^3-8mx+4dm}{(x+d)^4} 
&=\frac{x^3(x+4d)}{(x+d)^4}
+\frac{-8mx+4dm}{(x+d)^4} \\
&<\frac{-8mx+4dm}{(x+d)^4} 
<\frac{8m^\frac{4}{3}+4dm}{(-m^\frac{1}{3}+d)^4}. 
\end{aligned}
\end{equation*}
Assume $d=2a^2+4b^2$. Putting $Y=(a/b)^2$ yields 
\begin{equation*}
\begin{aligned}
\frac{8m^\frac{4}{3}+4dm}{(-m^\frac{1}{3}+d)^4} 
={{8\,\left(Y^3+16\right)\,\left(\left(Y^3+16\right)^{{{1}\over{3}}}+Y+2\right)}\over{\left(\left(Y^3+16\right)^{{{1}\over{3}}}-2\,Y-4\right)^4}}. 
\end{aligned}
\end{equation*}
We denote the right hand side by $g_{2,4}(Y)$. 
Then
\begin{equation*}
\frac{d}{dY}g_{2,4}(Y)
=-{{48\,\left(Y^2-8\right)\,\left(2\,Y\,\left(Y^3+16\right)^{{{2
 }\over{3}}}+4\,\left(Y^3+16\right)^{{{2}\over{3}}}+3\,Y^3+48\right)
 }\over{\left(Y^3+16\right)^{{{2}\over{3}}}\,\left(\left(Y^3+16
 \right)^{{{1}\over{3}}}-2\,Y-4\right)^5}}. 
\end{equation*}
Note that $\left(Y^3+16 \right)^{{{1}\over{3}}}-2\,Y-4 <0$, 
since $d-m^{1/3} >0$. 
So $g_{2,4}(Y)$ has a minimum at $Y=\sqrt{8}$ 
and a maximum at $Y=-\sqrt{8}$. 
Therefore 
\begin{equation*}
g_{2,4}(Y) 
\leq \max\left\{\lim_{Y\rightarrow 0}g_{2,4}(Y), 
\lim_{Y\rightarrow \infty} g_{2,4}(Y)\right\}. 
\end{equation*}
Since $g_{2,4}(0)=120.53163357\cdots$, 
$\lim_{Y\rightarrow \infty} g_{2,4}(Y)=16$, 
we have $z'(P')=g_{2,4}(Y) < 120.531634$. 

The case $d=3a^2+4b^2$ is similar and we have 
\begin{math}
z'(P')< 120.531634. 
\end{math}
\end{proof}

\begin{proof}
[Proof of Lemma \ref{remboundl}]
We use the notation at the beginning of 
the proof of Lemma \ref{remboundu}. 
Let $x=x(P)$, $u=x/d$ and $u_0=-m^{1/3}/d$.
Then $u \geq u_0 > -1$, since $x\geq -m^{1/3} >-d$.
Putting $Y=(a/b)^2$ with substitution $d=2a^2+4b^2$ yields 
\begin{equation*}
\begin{aligned}
& z'(P') =\frac{x^4+4dx^3-8mx+4dm}{(x+d)^4} 
=\frac{d^4u^4+4d^4u^3-8dmu+4dm}{(du+d)^4} \\
& = \frac
{2u^4(Y^3+6Y^2+12Y+8)+8u^3(Y^3+6Y^2+12Y+8)-2u(Y^3+16)+Y^3+16}
{2(u+1)^4(Y+2)^3}. 
\end{aligned}
\end{equation*}
We denote the last function by $f(u,Y)$. 
Computing the derivatives, we have 
\begin{equation*}
\begin{aligned}
\frac{\partial f}{\partial Y} 
& =-\frac{3(2u-1)(Y^2-8)}{(u+1)^4(Y+2)^4}, \\
\frac{\partial f}{\partial u} 
& =3\frac{(4u^2Y^3+uY^3-Y^3+24u^2Y^2+48u^2Y+32u^2+16u-16)}
{(u+1)^5(Y+2)^3} \\
& =12(Y^3+6Y^2+12Y+8)\frac{(u-u_1)(u-u_2)}
{(u+1)^5(Y+2)^3}, 
\end{aligned}
\end{equation*}
where 
\begin{equation*}
\begin{aligned}
u_1 & =-\frac{\sqrt{17Y^6+96Y^5+192Y^4+416Y^3+1536Y^2+3072Y+2304}+Y^3+16}
{8Y^3+48Y^2+96Y+64}, \\
u_2 & =\frac{\sqrt{17Y^6+96Y^5+192Y^4+416Y^3+1536Y^2+3072Y+2304}-Y^3-16}
{8Y^3+48Y^2+96Y+64}. 
\end{aligned}
\end{equation*}
Clearly $u_1 < 0 < u_2$ for any fixed $Y>0$, 
and the graph $f(u,Y)$ has the either of the two forms below. 
Therefore $f(u, Y) \geq \min\{f(u_0,Y), f(u_2,Y)\}$. 
\begin{figure}[htbp]
 \begin{minipage}{0.49\hsize}
  \begin{center}
   \includegraphics[width=76mm]{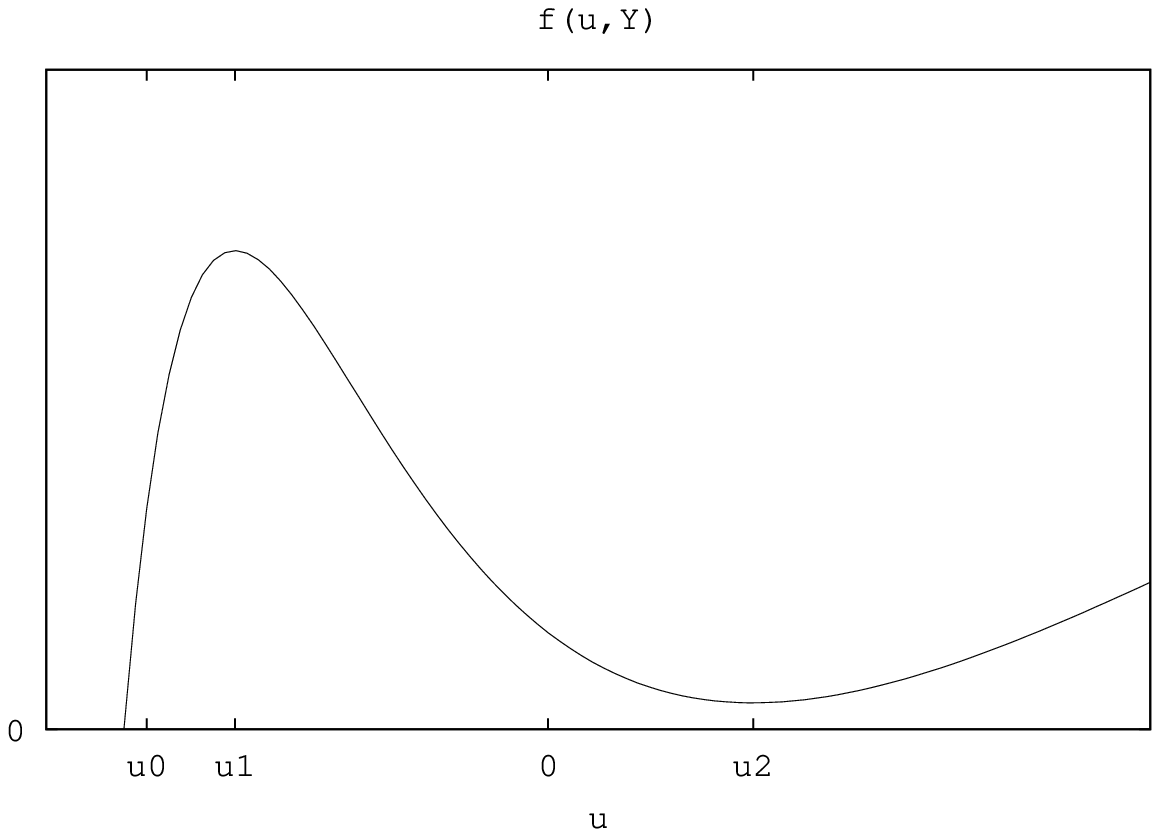}
  \end{center}
  \caption{}
  \label{fig:3}
 \end{minipage}
 \begin{minipage}{0.49\hsize}
  \begin{center}
   \includegraphics[width=76mm]{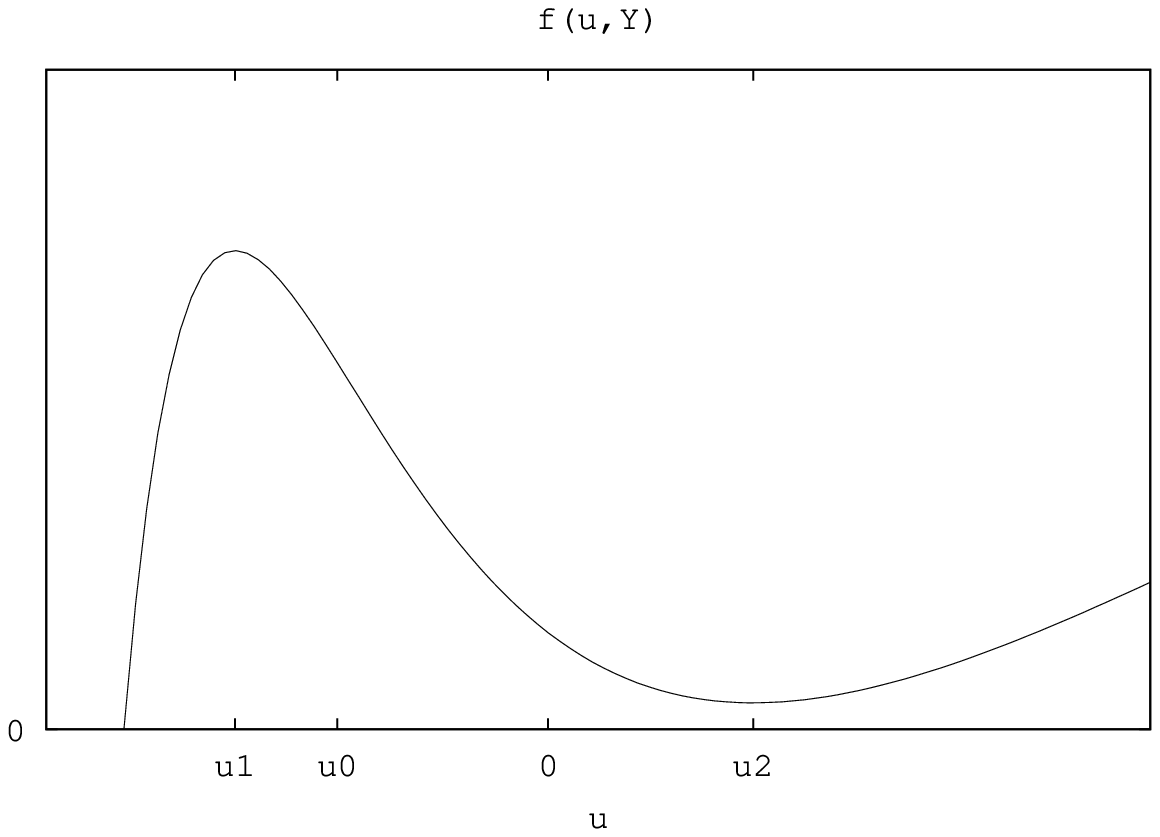}
  \end{center}
  \caption{}
  \label{fig:4}
 \end{minipage}
\end{figure}

At first we consider $f(u_0,Y)$. 
Since $d=2a^2+4b^2$, 
\begin{equation*}
\begin{aligned}
f(u_0,Y)=f(-m^{\frac{1}{3}}/d,Y)
= \left. 
\frac{d^4u^4+4d^4u^3-8dmu+4dm}{(du+d)^4} 
\right|_{x=-m^\frac{1}{3}/d} \\ 
= \frac{9m^\frac{4}{3}}{(d-m^\frac{1}{3})^4} 
= \frac{9(Y^3+16)^\frac{4}{3}}{(2Y+4-(Y^3+16)^{\frac{1}{3}})^4}
\geq 0.75725080\cdots. 
\end{aligned}
\end{equation*}
The last inequality follows from elementary calculus. 

Next we consider $f(u_2,Y)$. 
\begin{equation*}
\begin{aligned}
\frac{d f(u_2,Y)}{dY} 
& =\left. \frac{\partial f}{\partial u} \right|_{u=u_2} \cdot 
\frac{d u_2}{d Y}
+\left. \frac{\partial f}{\partial Y} \right|_{u=u_2} \cdot 
\frac{d Y}{d Y} \\
& =\left. \frac{\partial f}{\partial Y} \right|_{u=u_2} \\ 
& =-\frac{3(2u_2-1)(Y^2-8)}{(u_2+1)^4(Y+2)^4}. 
\end{aligned}
\end{equation*}
Now 
\begin{equation*}
\begin{aligned}
&2u_2-1 \\
&=\frac{2\sqrt{17Y^6+96Y^5+192Y^4+416Y^3+1536Y^2+3072Y+2304}-2Y^3-32}
{8Y^3+48Y^2+96Y+64}-1 \\
&=\frac{2\sqrt{17Y^6+96Y^5+\cdots}-(10Y^3+48Y^2+96Y+96)}
{8Y^3+48Y^2+96Y+64} <0, 
\end{aligned}
\end{equation*}
since 
\begin{equation*}
\begin{aligned}
\left(2\sqrt{17Y^6+96Y^5+\cdots}\right)^2-(10Y^3+48Y^2+96Y+96)^2 \\
=-32Y^6-576Y^5-3456Y^4-9472Y^3-12288Y^2-6144Y <0. 
\end{aligned}
\end{equation*}

So $f(u_2, Y) \geq f\left(u_2(\sqrt{8}), \sqrt{8}\right)$. 
Since 
$f\left(u_2(\sqrt{8}), \sqrt{8}\right)=0.06232685 \cdots$, 
$f(u_2, Y) > 0.062326$.
Therefore $z'(P')=f(u,Y) > 0.062326$.

The case $d=3a^2+4b^2$ is similar and we have 
$z'(P') > 0.03806854 \cdots$. 
\end{proof}
The commands for those computations by Maple are 
as follows. 

The commands for Lemma \ref{remboundu}. 
\begin{verbatim}
(8*m^(4/3)+4*d*m)/(x+d)^4;
subs(x=-m^(1/3),%);
subs(m=a^6+16*b^6,d=2*a^2+4*b^2,%);
g24:=simplify(subs(a=X,b=1,%)):
g24:=factor(simplify(subs(X=sqrt(Y),%)));
dg24:=factor(simplify(diff(%,Y)));
evalf(subs(Y=0,g24),16);
value(Limit(g24,Y=infinity));

(8*m^(4/3)+4*d*m)/(x+d)^4;
subs(x=-m^(1/3),%);
subs(m=a^6+16*b^6,d=3*a^2+4*b^2,%);
g34:=simplify(subs(a=X,b=1,%)):
g34:=factor(simplify(subs(X=sqrt(Y),%)));
dg34:=factor(simplify(diff(%,Y)));
evalf(subs(Y=0,g34),16);
value(Limit(g34,Y=infinity));
\end{verbatim}

The commands for Lemma \ref{remboundl}. 
\begin{itemize}
\item[(1)] $d=2a^2+4b^2$
\begin{verbatim}
(x^4+4*d*x^3-8*m*x+4*d*m)/(x+d)^4;
subs(x=u*d,%);
subs(m=a^6+16*b^6,d=2*a^2+4*b^2,%);
subs(a=X,b=1,%):
subs(X=sqrt(Y),%);
f24:=factor(simplify(%));
df24Y:=factor(diff(f24,Y));
df24u:=factor(diff(f24,u));
sol24:=solve(numer(df24u),u);
f24a:=subs(u=sol24[1],f24):
evalf(subs(Y=sqrt(8),f24a),16);

9*m^(4/3)/(d-m^(1/3))^4;
subs(m=a^6+16*b^6,d=2*a^2+4*b^2,%);
subs(a=X,b=1,%):
ff24:=factor(simplify(subs(X=sqrt(Y),%)));
factor(diff(%,Y));
evalf(subs(Y=sqrt(8),ff24),16);
\end{verbatim}
\end{itemize}

\begin{itemize}
\item[(2)] $d=3a^2+4b^2$
\begin{verbatim}
(x^4+4*d*x^3-8*m*x+4*d*m)/(x+d)^4;
subs(x=u*d,%);
subs(m=a^6+16*b^6,d=3*a^2+4*b^2,%);
subs(a=X,b=1,%):
subs(X=sqrt(Y),%);
f34:=factor(simplify(%));
df34Y:=factor(diff(f34,Y));
df34u:=factor(diff(f34,u));
sol34:=solve(numer(df34u),u);
f34a:=subs(u=sol34[1],f34):
evalf(subs(Y=sqrt(12),f34a),16);

9*m^(4/3)/(d-m^(1/3))^4;
subs(m=a^6+16*b^6,d=3*a^2+4*b^2,%);
subs(a=X,b=1,%):
ff34:=factor(simplify(subs(X=sqrt(Y),%)));
factor(diff(%,Y));
evalf(subs(Y=sqrt(12),ff34),16);
\end{verbatim}
\end{itemize}

\appendix
\section{}
\label{sec:magma code}
This is the code for Magma used in Remark \ref{rem:5-des}. 
By this code we can verify that 
$P_1$, $P_2$, $P_3$, 
$P_1 \pm P_2$, $P_2 \pm P_3$, $P_3 \pm P_1$, 
$P_1 \pm 2P_2$, $P_2 \pm 2P_3$, $P_3 \pm 2P_1$ 
$\notin 5E_{a,b}(\Q)$ for $5 \leq a \leq 6321, 3 \leq b \leq 3982$ 
in the situation of Theorem \ref{main}. 
\begin{verbatim}

for a in [5..6321] do
for b in [3..3982] do
if IsOdd(a*b) then
if Valuation(b, 3) eq 1 then
if Gcd(a,b) eq 1 then
m:=a^6+16*b^6;
if m lt 6.381*10^22 then
if IsSquarefree(m) then
E:=EllipticCurve([0,0,0,0,m]);
P1:=E![-a^2,4*b^3,1];
P2:=E![2*a*b,a^3+4*b^3,1];
P3:=E![-2*a*b,a^3-4*b^3,1];
P4:=P1+P2;
P5:=P2+P3;
P6:=P3+P1;
P7:=P1-P2;
P8:=P2-P3;
P9:=P3-P1;
P0a:=P1+2*P2;
P0b:=P1-2*P2;
P0c:=P2+2*P3;
P0d:=P2-2*P3;
P0e:=P3+2*P1;
P0f:=P3-2*P1;
if 
{
<
DivisionPoints(P1,5), DivisionPoints(P2,5), DivisionPoints(P3,5), 
DivisionPoints(P4,5), DivisionPoints(P5,5), DivisionPoints(P6,5), 
DivisionPoints(P7,5), DivisionPoints(P8,5), DivisionPoints(P9,5), 
DivisionPoints(P0a,5), DivisionPoints(P0b,5), DivisionPoints(P0c,5), 
DivisionPoints(P0d,5), DivisionPoints(P0e,5), DivisionPoints(P0f,5) 
>
} ne 
{
<[], [], [], [], [], [], [], [], [], [], [], [], [], [], []>
} then
print a, b;
end if;
end if;
end if;
end if;
end if;
end if;
end for;
end for;

\end{verbatim}

\section*{Acknowledgements}
The second author thanks his adviser Prof. Akihiko Yukie 
for careful reading and giving corrections.

\bibliographystyle{jplain}

\begin{thebibliography}{10}

\bibitem{cohen0}
{H. Cohen}.
\newblock {\em A Course in computational algebraic number theory}.
\newblock Springer-Verlag, 1993.

\bibitem{connel1}
{I. Connell}.
\newblock {\em Elliptic curve handbook}.
\newblock http://www.ucm.es/BUCM/mat/doc8354.pdf, 1996.

\bibitem{duquesne1}
{S. Duquesne}.
\newblock Elliptic curves associated with simplest quartic fields.
\newblock {\em J. Theor. Nombres Bordeaux}, Vol.~19, pp. 81--100, 2007.

\bibitem{FT}
{Y. Fujita and N. Terai}.
\newblock Generators for the elliptic curve $y^2=x^3-nx$.
\newblock {\em to appear in J. Theor. Nombres Bordeaux}.

\bibitem{greaves1}
{G. Greaves}.
\newblock Power-free values of binary forms.
\newblock {\em Quart. J. Math. Oxford (2)}, Vol.~43, pp. 45--65, 1992.

\bibitem{knappB1}
{A. Knapp}.
\newblock {\em Elliptic curves}.
\newblock Princeton Univ. Press, 1992.

\bibitem{magma}
{}.
\newblock {\em Magma}.
\newblock Computational Algebra Group, School of Mathematics and Statistics,
  University of Sydney, http://magma.maths.usyd.edu.au/magma/.

\bibitem{maple}
{}.
\newblock {\em Maple}.
\newblock Waterloo Maple Inc., http://www.maplesoft.com/products/maple/.

\bibitem{pari}
{}.
\newblock {\em PARI/GP}.
\newblock http://pari.math.u-bordeaux.fr/.

\bibitem{siksek1}
{S. Siksek}.
\newblock Infinite descent on elliptic curves.
\newblock {\em Rocky Mountain J. Math.}, Vol.~25, pp. 1501--1538, 1995.

\bibitem{sil1}
{J. H. Silverman}.
\newblock Computing heights on elliptic curves.
\newblock {\em Math. Comp.}, Vol.~51, pp. 339--358, 1988.

\bibitem{aec}
{J. H. Silverman}.
\newblock {\em The arithmetic of elliptic curves}.
\newblock Springer-Verlag, 1986.

\end{thebibliography}

\end{document}